\documentclass[runningheads,12pt]{llncs}
\usepackage [english]{babel}
\usepackage[utf8]{inputenc}
\usepackage[T1]{fontenc}

\usepackage{amssymb}
\usepackage{amsmath}
\usepackage{amsfonts}

\usepackage{graphicx}
\usepackage{subfigmat}
\usepackage{url}

\usepackage{latexsym} 
\usepackage{ wasysym}
\usepackage{newlfont}
\usepackage{leqno}
\usepackage[usenames]{color}
\usepackage{fancybox}
\usepackage{sidecap}
\usepackage{wrapfig}
\usepackage{sidecap}

  \usepackage{times}
  \usepackage{parskip}
  \usepackage{algorithm}
  \usepackage{algpseudocode}

  \usepackage{geometry}
\usepackage{stmaryrd}
\usepackage{xspace}

\numberwithin{equation}{section}

    \urldef{\mailsa}\path|ulrich.langer@ricam.oeaw.ac.at|
    \urldef{\mailse}\path|christoph.hofer@ricam.oeaw.ac.at|
    \newcommand{\keywords}[1]{\par\addvspace\baselineskip  
     \noindent\keywordname\enspace\ignorespaces#1}

\DeclareMathOperator{\mdiv}{div}
\DeclareMathOperator{\grad}{\nabla}

\newcommand{\for}{\mbox{for }}
\renewcommand{\ker}[1]{\text{ker}(#1)}
 
\newcommand{\genref}[2]{#2\,\ref{#1}}

\newcommand{\lemref}[1]{\genref{#1}{Lemma}}
\newcommand{\defref}[1]{\genref{#1}{Definition}}

\newcommand{\equref}[1]{(\ref{#1})} 
\newcommand{\secref}[1]{Section\,\ref{#1}}
\newcommand{\tabref}[1]{Table\,\ref{#1}}
\newcommand{\figref}[1]{Figure\,\ref{#1}}

\newcommand{\MatTwo}[4]
  {\begin{bmatrix}
    #1 & #2\\
    #3 & #4
\end{bmatrix}}
  
\newcommand{\VecTwo}[2]
  {\begin{bmatrix}
    #1 \\
    #2
\end{bmatrix}}

\begin{document}

\mainmatter


\title{DUAL-PRIMAL ISOGEOMETRIC TEARING AND INTERCONNECTING SOLVERS FOR MULTIPATCH DG-IGA EQUATIONS}
\titlerunning{dG-IETI-DP on multipatch dG-IgA}

\author{
Christoph Hofer$^1$
\and     Ulrich Langer$^1$ 
}
\authorrunning{C. Hofer,  U. Langer}

\institute{ $^1$ Johann Radon Institute for Computational and Applied Mathematics,\\ 
		  Austrian Academy of Sciences, 
		  Altenberger Str. 69, 4040 Linz, Austria.\\ 
\mailse\\
\mailsa \\
 }


%
\maketitle

\begin{abstract} 

In this paper we consider a new version of the dual-primal isogeometric tearing and interconnecting (IETI-DP) method for solving large-scale linear systems  of algebraic equations arising from discontinuous Galerkin (dG) isogeometric analysis of diffusion problems on multipatch domains  with non-matching meshes. The dG formulation is used to couple the local problems across patch interfaces. The purpose of this paper is to present this new method  and provide numerical examples indicating a polylogarithmic condition number bound for the preconditioned system and showing an incredible robustness with respect to large jumps in the diffusion coefficient across the interfaces. 
\end{abstract}



\section{Introduction}
\label{HL:Sec:Introduction}
Isogeometric Analysis (IgA) is an approach to approximate numerically the solution of  a partial differential equation (PDE) using the same basis functions for parametrizing the geometry and  for representing the solution. IgA was introduced by Huges, Cottrell and Bazilevs in \cite{HL:HughesCottrellBazilevs:2005a}, see also \cite{HL:BazilevsVeigaCottrellHughesSangalli:2006} for the first results on the numerical analysis of IgA, the monograph  \cite{HL:CotrellHughesBazilevs:2009} for a comprehensive presentation of the IgA,
and the recent survey article 
\cite{HL:BeiraodaVeigaBuffaSangalliVazquez:2014a}
on the mathematical analysis variational isogeometric methods.
There exists a wide variety of different basis functions, e.g.,  B-Splines, Non Uniform Rational B-Spline (NURBS), T-Splines, Hierarchical B-Splines (HB-Splines) and Truncated Hierarchical B-Splines (THB-Splines), see, e.g.,  
\cite{HL:CotrellHughesBazilevs:2009}, 
\cite{HL:BeiraodaVeigaBuffaSangalliVazquez:2014a},
\cite{HL:GiannelliJuettlerSpeleers:2012a} and
\cite{HL:GiannelliJuettlerSpeleers:2014a}.
%
A major advantage over the common finite element method (FEM) is more flexibility of h- and p-refinement, resulting also in $C^k,k\geq0$ continuous basis functions. However, due to the larger support of the basis functions, the resulting system matrices are denser and due to the more involved evaluation mechanism of B-Splines, NURBS, etc., we have to deal with much higher assembling times,
see, e.g., Section~8 in \cite{HL:BeiraodaVeigaBuffaSangalliVazquez:2014a} 
for a discussion of this issue and for relevant references. 
The low-rank tensor approximation technique,
proposed in \cite{HL:MantzaflarisJuettlerKhoromskijLanger:2015a},
is certainly a very smart and at the same time simple technology 
to overcome this bottleneck.

Beside matrix generation, the efficient solution of the linear systems
arising from IgA discretization of linear elliptic boundary value problem
or from the linearization of non-linear IgA equations 
turns out to be another bottleneck for the efficiency of the IgA technology.
In this paper, we consider a non-overlapping domain decomposition method based on 
the Finite Element Tearing and Interconnecting (FETI) for IgA, 
called IsogEometric Tearing and Interconnecting (IETI). 
In particular, we focus on the dual primal variant (IETI-DP), introduced in \cite{HL:KleissPechsteinJuettlerTomar:2012a}.
A comprehensive theoretical analysis of the FETI-DP and the equivalent  Balancing Domain Decomposition by Constraints (BDDC) method can be found in
the monographs
\cite{HL:ToselliWidlund:2005a} and \cite{HL:Pechstein:2013a}
where the reader also find the references
to the corresponding original papers.
The theoretical analysis of the IETI-DP and BDDC method was initiated in \cite{HL:VeigaChoPavarinoScacchi:2013a} and extended in \cite{HL:HoferLanger:2015a}. 
Based on the FE work in \cite{HL:DohrmannWidlund:2013a}, 
a recent improvement for the IgA BDDC preconditioner with a more advanced scaling technique, the so called \emph{deluxe scaling}, can be found in \cite{HL:VeigaPavarinoScacchiWidlundZampini:2014a}. Domain decomposition methods for IgA is currently a very active field of research. We mention developments in overlapping Schwarz methods, see, e.g. \cite{HL:VeigaChoPavarinoScacchi:2012a}, \cite{HL:VeigaChoPavarinoScacchi:2013b}, \cite{HL:BercovierSoloveichik:2015}, and isogeometric mortaring discretizations, see \cite{HL:HeschBetsch:2012a}.

In this paper, based on the IETI-DP mentioned above method, we 
are going 
to develop an efficient and robust solver for IgA systems arising from a discretization where the different patches are coupled with a discontinuous Galerkin (dG) method, called dG-IETI-DP. This setting is of special importance when considering non-matching meshes, see, e.g., \cite{HL:LangerToulopoulos:2014a}, and in case of non-matching interface parametrizations, resulting in gaps and overlaps, see, \cite{HL:HoferLangerToulopoulos_2015a} and \cite{HL:HoferToulopoulos_2015a}. The proposed method is based on the 
corresponding
version for finite elements (FE) proposed in \cite{HL:DryjaGalvisSarkis:2013a} and \cite{HL:DryjaSarkis:2014a}, 
where a rigorous analysis for 2D and 3D proves the same properties 
as for the classical FETI-DP, see also \cite{HL:DryjaGalvisSarkis:2007a} for an analysis of the corresponding BDDC preconditioner.

In the present paper, we consider the following 
second-order elliptic boundary value problem 
in a bounded Lipschitz domain $\Omega\subset \mathbb{R}, d\in\{2,3\}$,
as a typical model problem:
Find $u: \overline{\Omega} \rightarrow \mathbb{R}$ such that\\ 
%
%
\begin{equation}
  \label{equ:ModelStrong}
  - \mdiv(\alpha \grad u) = f \; \text{in } \Omega,\;
  u = 0 \; \text{on } \Gamma_D,
  \;\text{and}\;
  \alpha \frac{\partial u}{\partial n} = g_N \; \text{on } \Gamma_N,
\end{equation}
with given, sufficient smooth data $f, g_N \text{ and } \alpha$, where  
the coefficient $\alpha$ is uniformly bounded from below and above
by some positive constants $\alpha_{min}$ and $\alpha_{max}$, respectively.
The boundary $\Gamma = \partial \Omega$ of the computational domain $\Omega$
consists of a Dirichlet part $\Gamma_D$ of positive boundary measure 
and a Neumann part $\Gamma_N$.
Furthermore, we assume that the Dirichlet boundary $\Gamma_D$ is always 
a union of complete domain sides (edges / face in 2d / 3d) 
which are uniquely defined in IgA.
Without loss of generality, we assume 
homogeneous
Dirichlet conditions. 
This can always be  obtained by homogenization.
By means of integration by parts, 
we arrive at
the weak formulation of \equref{equ:ModelStrong} 
which reads as follows:
Find $u \in V_{D} = \{ u\in H^1: \gamma_0 u = 0 \text{ on } \Gamma_D \}$
such that
\begin{align}
  \label{equ:ModelVar}
    a(u,v) = \left\langle F, v \right\rangle \quad \forall v \in V_{D},
\end{align}
where $\gamma_0$ denotes the trace operator. The bilinear form
$a(\cdot,\cdot): V_{D} \times V_{D} \rightarrow \mathbb{R}$
and the linear form $\left\langle F, \cdot \right\rangle: V_{D} \rightarrow \mathbb{R}$
are given by the expressions
\begin{equation*}
a(u,v) = \int_\Omega \alpha \nabla u \nabla v \,dx
\quad \mbox{and} \quad
\left\langle F, v \right\rangle = \int_\Omega f v \,dx + \int_{\Gamma_N} g_N v \,ds.
\end{equation*}

The remainder of this paper is organized as follows. 
\secref{sec:Preliminaries} 
gives
a short overview of the main principles of IgA,
 and presents the dG-IgA formulation. The dG-IETI-DP method is defined and discussed in \secref{sec:dG_IETI_DP}. 
The numerical results, presented in \secref{sec:numerical}, 
demonstrate the numerical behaviour of dG-IETI-DP method.
In particular, we study the influence of the mesh size $h$, the patch diameter $H$, 
the use of non-matching meshes  
quantified by the mesh size ratio $h^{(k)}/h^{(l)}$ across the patch faces, 
and the polynomial degree $p$ on the 
condition number of the preconditioned system and, thus, on the 
number of iterations.
Finally, in \secref{HL:Sec:Conclusions}, we 
draw some conclusions and give some outlook.


\section{Preliminaries}
\label{sec:Preliminaries}
In this section, we 
give an overview of the tools required to 
describe
the IETI-DP method for multipatch dG-IgA equations. A more comprehensive study of IgA, IETI-DP and related topics can be found in \cite{HL:HoferLanger:2015a}. 


\subsection{B-Splines and IgA}
Let $[0,1]$ be the unit interval, 
the vector $\Xi = \{\xi_{1}=0,\xi_{2},\ldots,\xi_{m}=1\}$ with non-decreasing real values $\xi_{i}$ forms a partition of $[0,1]$ and is called $\emph{knot vector}$. 
Given a knot vector $\Xi$, $p\in\mathbb{N}$ and $n=m-p-1$, we can define the B-Spline function via the following recursive formulation:
 \begin{align}
 \label{equ:BSpline0}
  \check{N}_{i,0}(\xi) &=\begin{cases}
                 1 & \mbox{if }\xi_{i}\leq \xi \leq \xi_{i+1}\\
                 0 & \mbox{otherwise}
                \end{cases},\\
  \label{equ:BSpline}
  N_{i,p}(\xi)&= \frac{\xi-\xi_{i}}{\xi_{i+p}-\xi_{i}}N_{i,p-1}(\xi) + \frac{\xi_{i+p+1}-\xi}{\xi_{i+p+1}-\xi_{i+1}}N_{i+1,p-1}(\xi),
 \end{align}
 where $i=1,\ldots,n$ and $p$ is called \emph{degree}. From this recursion, we can observe
 that $N_{i,p}$ is a piecewise polynomial of degree $p$. Furthermore, we only consider open knot vectors, i.e., the first and the last node is repeated $p$ times.
 
 Since we are considering ${d}$-dimensional problems, we need to extent the concept of B-Splines to the ${d}$-dimensional space, which is done via the tensor product. Let $(p_1,\ldots,p_{d})$ be a vector in $\mathbb{N}^{d}$, and let, for all $\iota=1,\ldots,{d}$, $\Xi^\iota$ be a knot vector. Furthermore, we denote 
the $i_\iota$ univariate B-Spline defined on the knot vector $\Xi^\iota$ by $N_{i_\iota,p}^\iota(\xi^\iota)$. 
Then the ${d}$-dimensional tensor product B-Spline (TB-Spline) is defined by
 \begin{align}
  N_{(i_1,\ldots,i_d),(p_1,\ldots,p_d)}(\xi) = \prod_{\iota=1}^{{d}}N_{i_j,p_\iota}^\iota(\xi^\iota).
 \end{align}
In order to avoid cumbersome notations, we will again denote the tensor product B-Spline by $N_{i,p}$ and interpret $i$ and $p$ as multi-indices. Additionally, we define the set 
of multi-indices $\mathcal{I}$ by
\begin{align*}
 \mathcal{I} := \{(i_1,\ldots,i_{d}): i_\iota \in \{1,\ldots,M_\iota\} \}.
\end{align*} 
Since the knot vector $\Xi$ provides a partition of $(0,1)^{d}$, 
called \emph{parameter domain} in the following , it introduces a mesh $\hat{\mathcal{Q}}$, and we will denote a mesh element by $\hat{Q}$, called \emph{cell}. 

Now we are in a position 
to describe our computational domain, called \emph{physical domain}, $\Omega = G((0,1)^{d})$ by means of the 
\emph{geometrical mapping} $G$ defined by
\begin{align*}
   G :\;& (0,1)^{d} \rightarrow \mathbb{R}^{{g}}\\
	    &G(\xi) := \sum_{i\in \mathcal{I}}P_iN_{i,p}(\xi).
\end{align*}
%
In practice, it is often 
necessary to describe the physical domain $\Omega$ by $N$ non overlapping domains $\Omega^{(k)}$, called \emph{patches}. Each $\Omega^{(k)}$ is the image of an associated geometrical mapping $G^{(k)}$, defined on the parameter domain $(0,1)^{d}$, i.e., $\Omega^{(k)} = G^{(k)}\left((0,1)^{d}\right) \; \for k = 1,\ldots,N$, and $\overline{\Omega} = \bigcup_{k=1}^{N} \overline{\Omega}^{(k)}.$

We denote the interface between the two patches $\Omega^{(k)}$ and $\Omega^{(l)}$ by $\Gamma^{(k,l)}$, and the collection of all interfaces by $\Gamma$, i.e.,
\begin{equation*}
\Gamma^{(k,l)} = \overline{\Omega }^{(k)}\cap\overline{\Omega}^{(l)} 
\quad \mbox{and} \quad
\Gamma := \bigcup_{l>k}\Gamma^{(k,l)}.
\end{equation*}
Furthermore, the boundary of the domain $\Omega$ is denoted by $\partial\Omega$. 
Note that the interface $\Gamma$ is sometimes called \emph{skeleton}.

The key point of IgA is to use the same basis functions for representing the geometry via the geometrical mapping also for generating the trial and test spaces. 
Therefore, we define the basis functions in the physical domain as $\check{N}_{i,p}:= N_{i,p} \circ G^{-1}$ and the discrete function space by 
 \begin{align}
 \label{equ:gVh}
   V_{h} = \text{span}\{\check{N}_{i,p}\}_{i\in{\mathcal{I}}} \subset H^1(\Omega).
 \end{align}
 Moreover, each function 
$u(x) = \sum_{i\in\mathcal{I}} u_i \check{N}_{i,p}(x)$ 
is associated with the vector $\boldsymbol{u} = (u_i)_{i\in\mathcal{I}}$. 
This map is known as \emph{Ritz isomorphism}. One usually writes this relation as 
$u_h \leftrightarrow \boldsymbol{u}$,
%
%
 and we will use it in the following without further comments. If we consider a single patch $\Omega^{(k)}$ of a multipatch domain $\Omega$, we will use the notation $V_{h}^{(k)},\check{N}_{i,p}^{(k)},N_{i,p}^{(k)}$ and $G^{(k)}$ with the analogous definitions.


\subsection{Conforming Galerkin IgA Scheme}
\label{sec:cG_IGA}

In conforming Galerkin IgA schemes, we use functions which are continuous across patch interfaces, i.e.
 \begin{align*}
  V_{h} = \{v\;| \;v|_{\Omega^{(k)}}\in V_{h}^{(k)}\}\cap H^1 (\Omega).
 \end{align*}
 The decomposition of the space $V_{h}$ into basis function associated to $\Gamma$ and to the interior of each patch plays an important role for deriving IETI methods. 
Let us define the spaces
 \begin{align*}
  V_{\Gamma,h} := \text{span}\{ \check{N}_{i,p}|\; i\in \mathcal{I}_B\} \subset H^1(\Omega) \quad \text{and}\quad V_{I,h}^{(k)} := V_{h}^{(k)} \cap H^1_0(\Omega^{(k)}),
  \end{align*}
where $\mathcal{I}_B$ denotes all indices of basis functions having support on $\Gamma$.
  We are now able to state the desired decomposition
\begin{align}
 V_{h} = \prod_{k =1}^{N}V_{I,h}^{(k)} \oplus \mathcal{H}\left(V_{\Gamma,h}\right),
 \end{align}
where $\mathcal{H}$ is the \emph{discrete spline harmonic extension}, see \cite{HL:HoferLanger:2015a} and references therein.

The Galerkin IgA scheme reads as follows:
Find $u_h \in V_{D,h}$ such that
  \begin{align}
  \label{equ:ModelDisc}
    a(u_h,v_h) = \left\langle F, v_h \right\rangle \quad \forall v_h \in V_{D,h},
  \end{align}
where $V_{D,h} \subset V_{D}$ is the space of all functions from $V_{h}$ 
which vanish on the Dirichlet boundary $\Gamma_D$. 
A basis for this space 
is given by
the B-Spline functions $\{\check{N}_{i,p}\}_{i\in\mathcal{I}_0}$, where $\mathcal{I}_0$ contains all indices of $\mathcal{I}$ which do not have a support on the Dirichlet boundary $\Gamma_D$. 
Hence, the Galerkin IgA scheme $\equref{equ:ModelDisc}$ is equivalent to 
the linear system of algebraic equations
\begin {align}
\label{equ:Ku=f}
  \boldsymbol{K} \boldsymbol{u} = \boldsymbol{f},
\end{align}
where 
$\boldsymbol{K} = (\boldsymbol{K}_{i,j})_{i,j\in {\mathcal{I}}_0}$
and
$\boldsymbol{f}= (\boldsymbol{f}_i)_{i\in {\mathcal{I}}_0}$
denote the stiffness matrix and the load vector, respectively,
with 
$ \boldsymbol{K}_{i,j} = a(\check{N}_{j,p},\check{N}_{i,p})$
and 
$\boldsymbol{f}_i = \left\langle F, \check{N}_{i,p} \right\rangle$,
and $\boldsymbol{u}$ is the vector representation of $u_h$ 
given by the IgA isomorphism.


\subsection{Discontinuous Galerkin IgA Scheme}
\label{sec:dg_scheme}
The main principle of the dG-IgA Scheme is to use again the 
spaces $V_{h}^{(k)}$ of continuous functions on each patch $\Omega^{(k)}$, 
whereas 
discontinuities are allowed 
across the patch interface.
The continuity of the function value and its normal derivative are then enforced in a weak sense by adding additional terms to the bilinear form. This situation is especially important when we consider non-matching grids on each patch.
For the remainder of this paper, we define
the dG-IgA space
\begin{align}
\label{equ:gVh_glob}
  V_{h}:= V_{h}(\Omega) := \{v\,| \,v|_{\Omega^{(k)}}\in V_{h}^{(k)}\},
\end{align}
where $V_{h}^{(k)}$ is defined as in \equref{equ:gVh}. 
We now follow the notation 
used in
\cite{HL:Dryja:2003a} and  \cite{HL:DryjaGalvisSarkis:2013a}. 
A comprehensive 
study of dG schemes for FE can be found in \cite{HL:Riviere:2008a} and \cite{HL:PietroErn:2012a}. For an analysis of the dG-IgA scheme, we refer to \cite{HL:LangerToulopoulos:2014a}.
%
Dirichlet boundary conditions can be handled in different ways. 
We can use the dG technique to incorporating them in a weak sense,
see, e.g., \cite{HL:Arnold:1982a} and \cite{HL:ArnoldBrezziCockburnMarini:2002a}. 
This methods for imposing  Dirichlet boundary conditions 
was already proposed by Nitsche \cite{HL:Nitsche:1971a}.
Another method consists in enforcing them in 
a strong sense via an $L^2$ projection and homogenization.
In this paper, 
for simplicity of presentation,
we will follow the latter one, where we assume that the given Dirichlet data can be represented exactly with B-Splines. Hence, we define $V_{D,h}$ as the space of all functions from $V_{h}$ which vanish on the Dirichlet boundary $\Gamma_D$. 
Furthermore, we denote the set of all indices $l$  
such that $\Omega^{(k)}$ and $\Omega^{(l)}$ have a common edge/face (2D/3D) $F^{(kl)}$ by ${\mathcal{I}}_{\mathcal{F}}^{(k)}$.
Having these definitions at hand, we can define the discrete problem based on the Symmetric Interior Penalty (SIP) dG formulation as follows:
Find $u_h \in V_{D,h}$ such that
  \begin{align}
  \label{equ:ModelDiscDG}
    a_h(u_h,v_h) = \left\langle F, v_h \right\rangle \quad \forall v_h \in V_{D,h},
  \end{align}
where
\begin{align*}
    a_h(u,v) &:= \sum_{k=1}^N a_e^{(k)}(u,v) \quad \text{and} \quad \left\langle F, v \right\rangle:=\sum_{k=1}^N \int_{\Omega^{(k)}}f v^{(k)} dx,\\
    a^{(k)}_e(u,v) &:=  a^{(k)}(u,v) + s^{(k)}(u,v) + p^{(k)}(u,v),
\end{align*}
and
\begin{align*}
a^{(k)}(u,v) &:= \int_{\Omega^{(k)}}\alpha^{(k)} \nabla u^{(k)} \nabla v ^{(k)} dx,\\
    s^{(k)}(u,v)&:= \sum_{l\in{\mathcal{I}}_{\mathcal{F}}^{(k)}} \int_{F^{(kl)}}\frac{\alpha^{(k)}}{2}\left(\frac{\partial u^{(k)}}{\partial n}(v^{(l)}-v^{(k)})+ \frac{\partial v^{(k)}}{\partial n}(u^{(l)}-u^{(k)})\right)ds,\\ 
    p^{(k)}(u,v)&:= \sum_{l\in{\mathcal{I}}_{\mathcal{F}}^{(k)}} \int_{F^{(kl)}}\frac{\delta \alpha^{(k)}}{2h^{(kl)}}(u^{(l)}-u^{(k)})(v^{(l)}-v^{(k)})\,ds.
\end{align*}
The notation $\frac{\partial}{\partial n}$ means the derivative 
in the direction of the outer normal vector,
$\delta$ is a positive sufficiently large penalty parameter,  and $h^{(kl)}$ is the harmonic average of the adjacent mesh sizes, i.e., $h^{(kl)}=2 h^{(k)} h^{(l)}/(h^{(k)} + h^{(l)})$.

We equip $V_{D,h}$ with the broken Sobolev norm 
\begin{align*}
 \left\|u\right\|_{dG}^2  = \sum_{k = 1}^N\left[\alpha^{(k)} \left\|\nabla u^{(k)}\right\|_{\Omega^{(k)}}^2 + \sum_{l\in{\mathcal{I}}_{\mathcal{F}}^{(k)}} \frac{\delta \alpha^{(k)}}{h^{(kl)}}\int_{F^{(kl)}} (u^{(k)} - u^{(l)})^2 ds\right].
\end{align*}

Furthermore,
we define the bilinear forms
\begin{align*}
 d_h(u,v) = \sum_{k =1}^N d^{(k)}(u,v) \quad \text{and} \quad d^{(k)}(u,v)= a_e^{(k)}(u,v) + p^{(k)}(u,v).
\end{align*}
for later use.
We note that $\left\|u_h\right\|_{dG}^2 = d_h(u_h,u_h)$.

We are now able to show existence and uniqueness of a solution to \equref{equ:ModelDiscDG}. The following Lemma is an equivalent statement of Lemma 2.1 in \cite{HL:DryjaGalvisSarkis:2013a} for IgA,
and the proof is based on the results in \cite{HL:LangerToulopoulos:2014a}.
\begin{lemma}
\label{lem:wellPosedDg}
Let $\delta$ be sufficiently large. 
Then there exist two positive constants $\gamma_0$ and $\gamma_1$ 
which are independent of $h^{(k)},H^{(k)},\delta,\alpha^{(k)}$ and $u_h$
such that the inequalities
 \begin{align}
 \label{equ:equivPatchNormdG}
  \gamma_0 d^{(k)}(u_h,u_h)\leq a_e^{(k)}(u_h,u_h) \leq \gamma_1 d^{(k)}(u_h,u_h), \quad 
  \forall u_h\in V_{D,h}
 \end{align}
are valid for all $k=1,2,\ldots,N$. Furthermore, we have  the inequalities
\begin{align}
 \label{equ:equivNormdG}
 \gamma_0 \left\|u_h\right\|_{dG}^2\leq a_h(u_h,u_h)\leq \gamma_1 \left\|u_h\right\|_{dG}^2, \quad 
  \forall u_h\in V_{D,h}.
\end{align}

\end{lemma}
\begin{proof}
 Rewriting the proofs of Lemma~4.6 and Lemma~4.7 in \cite{HL:LangerToulopoulos:2014a} for a single patch gives the desired inequalities \equref{equ:equivPatchNormdG}. 
In order to show the boundedness, we additionally need to apply the discrete inverse inequality $\left\|\nabla u_h\right\|_{L^2(F^{(kl))}}^2\leq C/h^{(k)} \left\|\nabla u_h\right\|_{L^2(\Omega^{(k))}}^2$, see, e.g., \cite{HL:EvansHughes:2013a}, to the term $\sum_{l\in {\mathcal{I}}_{\mathcal{F}}^{(k)}} \alpha^{(k)} h^{(k)}\left\|\nabla u_h\right\|_{L^2(F^{(kl))}}^2$ appearing in the bound of Lemma 4.7 in \cite{HL:LangerToulopoulos:2014a}. 
Then we easily arrive at the estimate
 \begin{align*}
  \sum_{l\in {\mathcal{I}}_{\mathcal{F}}^{(k)}} \alpha^{(k)} h^{(k)}\left\|\nabla u_h\right\|_{L^2(F^{(kl))}}^2 \leq  C \sum_{l\in {\mathcal{I}}_{\mathcal{F}}^{(k)}} \alpha^{(k)} \left\|\nabla u_h\right\|_{L^2(\Omega^{(k))}}^2.
 \end{align*}
Hence, the right-hand side can be bounded by $d^{(k)}(u_h,u_h)$. 
Formula \equref{equ:equivNormdG} immediately follows from \equref{equ:equivPatchNormdG}, which concludes the proof.
\qed \end{proof}

%
We note that, in \cite{HL:LangerToulopoulos:2014a}, the results were obtained 
for the Incomplete Interior Penalty (IIP) scheme. 
An extension to SIP-dG and the use of  harmonic averages for $h$ and/or $\alpha$ 
are discussed in Remark~3.1 in \cite{HL:LangerToulopoulos:2014a},
see also \cite{HL:LangerMantzaflarisMooreToulopoulos:2015b}.

A direct implication of \equref{equ:equivNormdG} is the well posedness of the discrete problem \equref{equ:ModelDisc} by the Theorem of Lax-Milgram. The consistency of the method together with interpolation estimates for B-spline quasi-interpolant lead to the following a-priori error estimate, as established in \cite{HL:LangerToulopoulos:2014a}.
%
%

\begin{theorem}
 Let $u\in H^1(\Omega)\cap \prod_{k=1}^N W^{l+1,q}(\Omega^{(k)})$
with $q\in (\min\{1,2d/(d+2l)\},2]$ and some integer $l\geq 1$, 
solves \equref{equ:ModelVar}, and let $u_h\in V_{D,h}$ solves the discrete problem \equref{equ:ModelDiscDG}. Then the discretization $u-u_h$ satisfies the estimate
 \begin{align*}
  \left\|u-u_h\right\|_{dG}^2\leq  \sum_{k=1}^N C^{(k)}\left(\left({h^{(k)}}\right)^{2r}+\sum_{j \in{\mathcal{I}}_{\mathcal{F}}^{(k)}}\alpha^{(k)} \frac{h^{(k)}}{h^{(j)}}\left({h^{(k)}}\right)^{2r}\right),
 \end{align*}
where 
$r = \min \{l+ (\frac{d}{2}-\frac{d}{q}),p\}$, and 
$C^{(k)}$
is a positive constant  which depends on 
$p$, $\|u\|_{W^{l+1,q}(\Omega^{(k)})}$,
and
$\max_{l_0\le l+1}\|\nabla^{l_0} G^{(k)}\|_{L^{\infty}(\Omega^{(k)})}$,
but not on $h$.
Here $W^{l+1,q}(\Omega^{(k)})$ denotes the Sobolev space of all functions from 
the space $L^q(\Omega^{(k)})$ such that all weak derivatives 
up to
order $l+1$
belong to $L^q(\Omega^{(k)})$ as well.
\end{theorem}

As explained in \secref{sec:cG_IGA}, we choose the B-Spline functions $\{\check{N}_{i,p}\}_{i\in\mathcal{I}_0}$ as basis for the space $V_{h}$, see \equref{equ:gVh_glob}, where $\mathcal{I}_0$ contains all indices of $\mathcal{I}$, which do not have a support on the Dirichlet boundary. Hence, the dG-IgA scheme $\equref{equ:ModelDiscDG}$ is equivalent to the system of linear equations
\begin {align}
\label{equ:Ku=f_DG}
  \boldsymbol{K} \boldsymbol{u} = \boldsymbol{f},
\end{align}
where 
$\boldsymbol{K} = (\boldsymbol{K}_{i,j})_{i,j\in {\mathcal{I}}_0}$
and
$\boldsymbol{f}= (\boldsymbol{f}_i)_{i\in {\mathcal{I}}_0}$
denote the stiffness matrix and the load vector, respectively,
with 
$ \boldsymbol{K}_{i,j} = a(\check{N}_{j,p},\check{N}_{i,p})$
and 
$\boldsymbol{f}_i = \left\langle F, \check{N}_{i,p} \right\rangle$, and
$\boldsymbol{u}$ is the vector representation of $u_h$. 

------------------------------------------------------------------------------------

\section{IsoGeometric Tearing and Interconnecting for multipatch dG}
\label{sec:dG_IETI_DP}

Let us consider a multipatch domain, where the interfaces are geometrically matching, 
but not the meshes, i.e. the meshes can be different on different patches.
Therefore,
the considered solution and test space do not provide continuity across patch interfaces. 
Hence, we cannot enforce continuity of the solution by means of the jump operator as in the conforming IETI-DP. 
As proposed in \cite{HL:DryjaGalvisSarkis:2013a}, the remedy will be to introduce an additional layer of dofs on the interfaces and enforce continuity between the different layers. The considered method can then be seen as a conforming IETI-DP on an extended grid of dofs.  We will follow the derivation presented in \cite{HL:DryjaGalvisSarkis:2013a} with adopted notations. In the following, let $V_{h}$ be the dG-IgA space
which fulfils the Dirichlet boundary conditions as defined in \secref{sec:dg_scheme} and we denote by $\{\check{N}_{i,p}\}_{i\in{\mathcal{I}}}$ the corresponding B-Spline basis. 
%
%
\subsection{Basic setup and local space description}
\label{sec:LocSpaces}

As already introduced  in \secref{sec:dg_scheme}, let ${\mathcal{I}}_{\mathcal{F}}^{(k)}$ be the set of all indices 
$l$ such that $\Omega^{(k)}$ and $\Omega^{(l)}$ share a common edge/face. We may denote ${\mathcal{I}}_{\mathcal{F}}^{(k)}$ by $\mathcal{E}^{(k)}$ when considering 2D domains and by $\mathcal{F}^{(k)}$ for 3D domains. 
 If we consider 3D objects,  we additionally define $\overline E^{(klm)}$ as the edge shared by the patches $\Omega^{(k)}$, $\Omega^{(l)}$ and $\Omega^{(m)}$, i.e., $E^{(klm)} = \partial F^{(kl)} \cap \partial F^{(km)}$ for $l\in\mathcal{F}^{(k)},m\in\mathcal{F}^{(k)}$. The set of all indices $(l,m)$ of $\Omega^{(l)}$ and $\Omega^{(m)}$, such that $\overline E^{(klm)}$ is an edge of patch $\Omega^{(k)}$ is denoted by $\mathcal{E}^{(k)}$.  Note, although $ \overline{F}_{lk} \subset \partial\Omega^{(l)}$ and  $\overline{F}_{k l}\subset \partial\Omega^{(k)}$ are geometrically the same, they are treated as different objects. The same applies for edges $\overline E^{(klm)},\overline E^{(lkm)}$ and $\overline E^{(mkl)}$. In order to keep the presentation of the method simple, we assume that the considered patch $\Omega^{(k)}$ does not touch the Dirichlet boundary. The other case can be handled in an analogous way.
 
As already introduced above, the computational domain $\Omega$ is given by $\overline{\Omega} = \bigcup_{k=1}^{N} \overline{\Omega}^{(k)}$, where $\Omega^{(k)} = G^{(k)}\left((0,1)^{d}\right) \; \for k = 1,\ldots,N$, and the interface by $\Gamma^{(k)} = \overline{\partial\Omega^{(k)}\backslash\partial\Omega}$. 
For each patch $\Omega^{(k)}$, we introduce  its extended version $\Omega^{(k)}_e$ via the union with all neighbouring interfaces $\overline{F}_{lk}\subset\partial\Omega^{(l)}$:
\begin{align*}
  \overline{\Omega}^{(k)}_e := \overline{\Omega}^{(k)} \cup \{\bigcup_{l\in{\mathcal{I}}_{\mathcal{F}}^{(k)}} \overline{F}^{(lk)}\}.
\end{align*}
Moreover, the extended interface $\Gamma^{(k)}_e$ is given by the union of $\Gamma^{(k)}$ with all neighbouring interfaces $\overline{F}_{lk}\subset\partial\Omega^{(k)}$:
\begin{align*}
 \Gamma^{(k)}_e := \Gamma^{(k)} \cup \{\bigcup_{l\in{\mathcal{I}}_{\mathcal{F}}^{(k)}} \overline{F}^{(lk)}\}.
\end{align*}
Based on the definitions above, we can introduce 
\begin{align*}
\overline{\Omega}_e = \bigcup_{k=1}^{N} \overline{\Omega}_e^{(k)}, \quad  \Gamma = \bigcup_{k = 1}^N \Gamma^{(k)} \text{ and}\quad \Gamma_e = \bigcup_{k = 1}^N\Gamma^{(k)}_e.
\end{align*}
The next step is to describe appropriate discrete function spaces to reformulate \equref{equ:ModelDiscDG} in order to treat the new formulation in the spirit of the conforming IETI-DP method.  We start with a description of the discrete function spaces for a single patch. 

As defined in $\equref{equ:gVh}$, let $V_{h}^{(k)}$ be the discrete function space 
defined on the patch $\Omega^{(k)}$. 
Then we define the corresponding function space for the extended patch $\Omega^{(k)}_e$ by
\begin{align*}
  V_{h,e}^{(k)} := V_{h}^{(k)} \times \prod_{l\in{\mathcal{I}}_{\mathcal{F}}^{(k)}}V_{h}^{(k)}(\overline{F}^{(lk)}),
\end{align*}
where $V_{h}^{(k)}(\overline{F}^{(lk)}) \subset V_{h}^{(l)}$ is given by
\begin{align*}
 V_{h}^{(k)}(\overline{F}^{(lk)}) := \text{span}\{\check{N}_{i,p}^{(l)} \,|\, \text{supp}\{\check{N}_{i,p}^{(l)}\}\cap\overline{F}^{(lk)} \neq \emptyset\}.
\end{align*}
According to the notation introduced in \cite{HL:DryjaGalvisSarkis:2013a}, we will represent a function $u^{(k)} \in V_{h,e}^{(k)} $ as
\begin{align}
 u^{(k)} = \{(u^{(k)})^{(k)}, \{(u^{(k)})^{(l)}\}_{l\in {\mathcal{I}}_{\mathcal{F}}^{(k)}}\},
\end{align}
where $(u^{(k)})^{(k)}$ and $(u^{(k)})^{(l)}$  are the restrictions of $u^{(k)}$ to $\Omega^{(k)}$ and $\overline{F}^{(lk)}$, respectively. Moreover, we introduce an additional representation of $u^{(k)}\in V_{h,e}^{(k)} $, as $u^{(k)} = (u^{(k)}_I, u^{(k)}_{B_e})$, where 
\begin{align*}
 u^{(k)}_I\in V_{I,h}^{(k)}:=V_{h}^{(k)} \cap H^1_0(\Omega^{(k)}),
\end{align*}
and
\begin{align*}
  u^{(k)}_{B_e} \in W^{(k)}:=\text{span}\{\check{N}_{i,p}^{(l)}\,|\, \,\text{supp}\{\check{N}_{i,p}^{(l)}\}\cap\Gamma^{(k)}_e \neq \emptyset \text{ for } l\in{\mathcal{I}}_{\mathcal{F}}^{(k)}\cup\{k\}\}.
\end{align*}
This provides a representation of $V_{h,e}^{(k)} $ 
in the form of
$V_{I,h}^{(k)} \times W^{(k)}$. 

------------------------------------------------------------------------------------

\subsection{ Schur complement and discrete harmonic extensions}
\label{sec:DHE_Schur}

We note that the bilinear form $a^{(k)}_e(\cdot,\cdot)$ is defined on the space $V_{h,e}^{(k)}\times V_{h,e}^{(k)}$, since it requires function values of the neighbouring patches $\Omega^{(l)},l\in{\mathcal{I}}_{\mathcal{F}}^{(k)}$. Therefore, it depicts a matrix representation $\boldsymbol{K}_e^{(k)}$ 
satisfying the identity
\begin{align*}
 a^{(k)}_e(u^{(k)},v^{(k)}) = (\boldsymbol{K}_e^{(k)} \boldsymbol{u},\boldsymbol{v})_{l_2}\quad \text{for } u^{(k)},v^{(k)} \in V_{h,e}^{(k)},
\end{align*}
where $\boldsymbol{u}$ and $\boldsymbol{v}$ denote the  vector representation of $u^{(k)}$ and $v^{(k)}$, respectively. By means of the representation $V_{I,h}^{(k)} \times W^{(k)}$ for $V_{h,e}^{(k)} $, we can partition the matrix $\boldsymbol{K}^{(k)}_e$ as
\begin{align}
\label{equ:loc_K}
 \boldsymbol{K}^{(k)}_e
 \begin{bmatrix}
  \boldsymbol{K}^{(k)}_{e,II} & \boldsymbol{K}^{(k)}_{e,IB_e} \\
  \boldsymbol{K}^{(k)}_{e,B_e I} & \boldsymbol{K}^{(k)}_{e,B_e B_e }.
 \end{bmatrix}
\end{align}
This 
enables 
us to define the Schur complement of $\boldsymbol{K}^{(k)}_e$ with respect to $W^{(k)}$ as
\begin{align}
\label{equ:loc_Schur}
 \boldsymbol{S}^{(k)}_e := \boldsymbol{K}^{(k)}_{e,B_e B_e } - \boldsymbol{K}^{(k)}_{e,B_e I}\left(\boldsymbol{K}^{(k)}_{e,II}\right)^{-1}\boldsymbol{K}^{(k)}_{e,IB_e}.
\end{align}
We denote 
the corresponding bilinear form 
by $s^{(k)}_e(\cdot,\cdot)$,
and 
the corresponding operator
by $S^{(k)}_e: W^{(k)}\to{W^{(k)}}^*$, i.e.
\begin{align*}
 (\boldsymbol{S}^{(k)}_e \boldsymbol{u}_{B_e}^{(k)},\boldsymbol{v}_{B_e}^{(k)})_{l^2} = \langle S^{(k)}_e u_{B_e}^{(k)},u_{B_e}^{(k)} \rangle = s^{(k)}_e(u_{B_e}^{(k)},u_{B_e}^{(k)}), \quad \forall u_{B_e}^{(k)},u_{B_e}^{(k)}\in W^{(k)}.
\end{align*}
The Schur complement has the property that
\begin{align}
\label{equ:SchurMin}
 \langle S^{(k)}_e u_{B_e}^{(k)},u_{B_e}^{(k)}\rangle = \min_{w^{(k)}= (w_I^{(k)},w_{B_e}^{(k)})\in V_{h,e}^{(k)}} a^{(k)}_e(w^{(k)},w^{(k)}),
\end{align}
such that $w_{B_e}^{(k)} = u_{B_e}^{(k)}$ on $\Gamma^{(k)}_e$. 
We define the \emph{discrete NURBS harmonic extension} $\mathcal{H}^{(k)}_e$ (in the sense of $a^{(k)}_e(\cdot,\cdot)$) for patch $\Omega^{(k)}_e$ by
\begin{align}
\label{def:DHEext_loc}
\begin{split}
 \mathcal{H}^{(k)}_e&: W^{(k)} \to V_{h,e}^{(k)}:\\
&\begin{cases}
 \text{Find }\mathcal{H}^{(k)}_e{u_{B_e}}\in V_{h,e}^{(k)}: & \\
  \quad a^{(k)}_e(\mathcal{H}^{(k)}_e{u_{B_e}},u^{(k)})=0 \quad &\forall u^{(k)}\in V_{I,h}^{(k)},\\
   \quad \mathcal{H}^{(k)}_e{u_{B_e}}_{|\Gamma^{(k)}} = {u_{B_e}}_{|\Gamma^{(k)}}, &
\end{cases}
\end{split}
\end{align}
where 
$V_{I,h}^{(k)}$ 
is here interpreted as subspace of $V_{h,e}^{(k)}$ with vanishing function values on $\Gamma_e^{(k)}$. One can show that the minimizer in \equref{equ:SchurMin} is given by $\mathcal{H}^{(k)}_e{u_{B_e}}$. 
In addition, we introduce the \emph{standard discrete NURBS harmonic extension} $\mathcal{H}^{(k)}$ (in the sense of $a^{(k)}(\cdot,\cdot)$) of $u^{(k)}_{B_e}$ as follows:
\begin{align}
\label{def:DHE_loc}
\begin{split}
 \mathcal{H}^{(k)}&: W^{(k)} \to V_{h,e}^{(k)}:\\
&\begin{cases}
 \text{Find }\mathcal{H}^{(k)}{u_{B_e}}\in V_{h,e}^{(k)}: & \\
  \quad a^{(k)}(\mathcal{H}^{(k)}{u_{B_e}},u^{(k)})=0 \quad &\forall u^{(k)}\in V_{I,h}^{(k)},\\
   \quad \mathcal{H}^{(k)}{u_{B_e}}_{|\Gamma^{(k)}} = {u_{B_e}}_{|\Gamma^{(k)}}, &
\end{cases}
\end{split}
\end{align}
where $V_{I,h}^{(k)}$ is the same space as in \equref{def:DHEext_loc}, 
and $a^{(k)}(\cdot,\cdot)$ 
is a
bilinear form on the space 
$V_{h,e}^{(k)}\times V_{h,e}^{(k)}$. 
The crucial point
is to show equivalence in the energy norm $d_h(u_h,u_h)$ between functions, 
which are discrete harmonic in the sense of $\mathcal{H}^{(k)}_e$ and $\mathcal{H}^{(k)}$. 
This property is summarized in the following Lemma, 
cf. also Lemma~3.1 in \cite{HL:DryjaGalvisSarkis:2013a}.
%
%
\begin{lemma}
\label{lem:equDiscHarmonic}
There exists a positive constant 
which  is independent of $\delta, h^{(k)}, H^{(k)}, \alpha^{(k)}$ and $u_{B_e}^{(k)}$
such that the inequalities
 \begin{align}
  d^{(k)}(\mathcal{H}^{(k)}{u_{B_e}},\mathcal{H}^{(k)}{u_{B_e}})\leq d^{(k)}(\mathcal{H}_e^{(k)}{u_{B_e}},\mathcal{H}_e^{(k)}{u_{B_e}})\leq C d^{(k)}(\mathcal{H}^{(k)}{u_{B_e}},\mathcal{H}^{(k)}{u_{B_e}}),
 \end{align}
hold for all $u_{B_e}^{(k)}\in W^{(k)}$.
\end{lemma}
\begin{proof}
 The proof is identical to that one presented  in \cite{HL:DryjaGalvisSarkis:2007a} for Lemma~4.1 
up to the point where we have to use the discrete trace inequality 
 \begin{align*}
  \left\|u_h\right\|_{L^2(\partial\Omega^{(k))}}^2 \leq Ch^{-1}\left\|u_h\right\|_{L^2(\Omega^{(k))}}^2, \quad \forall u_h\in V_{h}^{(k)},
 \end{align*}
for IgA function spaces,
see, e.g., \cite{HL:EvansHughes:2013a}.
\qed \end{proof}
The subsequent statement immediately follows from \lemref{lem:wellPosedDg} and \lemref{lem:equDiscHarmonic}, see also \cite{HL:DryjaGalvisSarkis:2013a}.
%
%
\begin{corollary}
The spectral equivalence inequalities
  \begin{align}
  C_0 d^{(k)}(\mathcal{H}^{(k)}{u_{B_e}},\mathcal{H}^{(k)}{u_{B_e}})\leq a_e^{(k)}(\mathcal{H}_e^{(k)}{u_{B_e}},\mathcal{H}_e^{(k)}{u_{B_e}})\leq C_1 d^{(k)}(\mathcal{H}^{(k)}{u_{B_e}},\mathcal{H}^{(k)}{u_{B_e}}),
 \end{align}
hold for all $u_{B_e}^{(k)}\in W^{(k)}$,
where the constants $C_0$ and $C_1$ are independent of 
$\delta, h^{(k)}, H^{(k)}, \alpha^{(k)}$ and $u_{B_e}^{(k)}$.
\end{corollary}

%
%

\subsection{Global space description}

Based on the definitions of the local spaces in \secref{sec:LocSpaces}, we can introduce the corresponding spaces 
\begin{align*}
 V_{h,e} := \{ v \,|\, v^{(k)}\in V_{h,e}^{(k)}, k\in\{1,\ldots,N\}\}.
\end{align*}
for the whole extended domain $\Omega_e$.
Additionally, we need a description of the global extended interface spaces
\begin{align*}
 W := \{ v_{B_e} \,|\, v_{B_e}^{(k)}\in W^{(k)}, k\in\{1,\ldots,N\}\} = \prod_{k=1}^N W^{(k)}.
\end{align*}
We note that according to \cite{HL:DryjaGalvisSarkis:2013a}, we will also interpret this space  as subspace of $V_{h,e}$, where its functions are discrete harmonic in the sense of $\mathcal{H}^{(k)}_e$ on each $\Omega^{(k)}$. For completeness, we define the discrete NURBS harmonic extension in the sense of $\sum_{k=1}^N a^{(k)}_e(\cdot,\cdot)$ and $\sum_{k=1}^N a^{(k)}(\cdot,\cdot)$ for $W$ as $\mathcal{H}_e u =\{\mathcal{H}_e^{(k)} u^{(k)}\}_{k=1}^N$ and $\mathcal{H}_e u =\{\mathcal{H}^{(k)} u^{(k)}\}_{k=1}^N$, respectively.

The goal is to reformulate \equref{equ:ModelDiscDG} and \equref{equ:Ku=f_DG} in terms of the extended domain $\Omega_e$. In order to achieve this, we need a coupling of the now independent interface dofs. In the context of tearing and interconnecting methods, we need a ``continuous'' subspace $\widehat{W}$ of $W$ such that $\widehat{W}$ is equivalent to $V_{\Gamma,h}$, i.e., $\widehat{W} \equiv V_{\Gamma,h}$. Since the space $V_{\Gamma,h}$ consists of functions which are discontinuous across the patch interface, the common understanding of continuity makes no sense. We follow the way in \cite{HL:DryjaGalvisSarkis:2007a}, providing an appropriate definition of continuity in the context of the spaces $\widehat{W}, W, V_{\Gamma,h}, V_{h,e}$ and $V_{h}$.
\begin{definition}
 We say that $u\in V_{h,e}$ is continuous on $\Gamma_e$ if 
the relations
 \begin{align}
 \label{def:continuous_1}
  (\boldsymbol{u}^{(k)})^{(k)}_i =  (\boldsymbol{u}^{(l)})^{(k)}_j \quad\forall (i,j)\in B_e(k,l), \;\forall l\in{\mathcal{I}}_{\mathcal{F}}^{(k)},
 \end{align}
and
 \begin{align}
 \label{def:continuous_2}
  (\boldsymbol{u}^{(k)})^{(l)}_i =  (\boldsymbol{u}^{(l)})^{(l)}_j \quad\forall (i,j)\in B_e(l,k), \;\forall l\in{\mathcal{I}}_{\mathcal{F}}^{(k)}.
 \end{align}
hold for all $k\in\{1,\ldots,N\}$.
 We denote 
the set of index pairs $(i,j)$ such that the i-th basis function in $V_{h}^{(k)}$ can be identified with the j-th basis function in $V_{h}^{(l)}(\overline{F}^{(kl)})$
by $B_e(k,l)$. 
We note that $B_e(k,l) \neq B_e(l,k)$.
 Moreover, 
$\widehat{V}_{h,e}$ 
denotes
the subspace of continuous functions on $\Gamma_e$ of $V_{h,e}$. Furthermore, $\widehat{V}_{h,e}$ can be identified with $V_{h}$.
\end{definition}
The operator $B : W \to U^*:=\mathbb{R}^{{\Lambda}}$, 
which realizes constraints \equref{def:continuous_1} and \equref{def:continuous_2}
in the form
  \begin{align*}
    Bu = 0,
  \end{align*}
is called \emph{jump operator}.
The space of all functions in $W$
which belong to the kernel of $B$ 
is denoted by $\widehat{W}$, and can be identified with $V_{\Gamma,h}$, i.e.
  \begin{align*}
   \widehat{W} = \{ w\in W|\, Bw =0 \} \equiv V_{\Gamma,h}.
  \end{align*}
  Furthermore, we define the restriction of $\widehat{W}$ to $\Omega^{(k)}_e$ by $\widehat{W}^{(k)}$.
\begin{remark}
 According to \cite{HL:DryjaGalvisSarkis:2013a}, the space $\widehat{W}$ can also be interpreted as the set of all functions in $\widehat{V}_{h,e}$ which are discrete harmonic in the sense of $\mathcal{H}_e$. 
\end{remark}
We introduce the set of patch vertices $\mathcal{V}$ and the corresponding extended set $\mathcal{V}_e$, given by
\begin{align*}
 \mathcal{V}_e^{(k)} = \mathcal{V}^{(k)}\cup \{\bigcup_{l\in{\mathcal{I}}_{\mathcal{F}}^{(k)}}\partial F^{(lk)} \}, \text{ where } \quad \mathcal{V}^{(k)} = \{\bigcup_{l\in{\mathcal{I}}_{\mathcal{F}}^{(k)}}\partial F^{(kl)}\},
 \end{align*}
 for 2D domains and by 
 \begin{align*}
  \mathcal{V}_e^{(k)} = \mathcal{V}^{(k)}\cup \{\bigcup_{(l,m)\in\mathcal{E}^{(k)}}\partial E^{(lkm)}\cup \partial E^{(mkl)} \}, \text{where } \quad\mathcal{V}^{(k)} = \{\bigcup_{(l,m)\in\mathcal{E}^{(k)}}\partial E^{(klm)}\}, 
\end{align*}
for 3D domains.
%
The set $\mathcal{V}$ and $\mathcal{V}_e$ is then given by the union of all $\mathcal{V}^{(k)}$ and $\mathcal{V}^{(k)}_e$, respectively. Moreover, we denote by $\mathcal{V}_e^{(kl)}\subset\mathcal{V}_e$ all vertices which belong to the interface $F^{(kl)}$.

Now we are in the position to reformulate \equref{equ:Ku=f_DG} in terms of $\widehat{V}_{h,e}$,
leading to the system
\begin{align}
\label{equ:hKu=hf}
 \widehat{\boldsymbol{K}}_e \boldsymbol{u}_e = \widehat{\boldsymbol{f}}_e,
\end{align}
where the matrix $\widehat{\boldsymbol{K}}_e$ is given by the assembly of the patchwise matrices $\boldsymbol{K}_e^{(k)}$, i.e.
\begin{align}
 \widehat{\boldsymbol{K}}_e = \sum_{k=1}^{N}\boldsymbol{A}_{\Omega^{(k)}_e}\boldsymbol{K}_e^{(k)}\boldsymbol{A}^T_{\Omega^{(k)}_e}\text{ and}\quad  \widehat{\boldsymbol{f}}_e=\sum_{k=1}^{N}\boldsymbol{A}_{\Omega^{(k)}_e}\boldsymbol{f}^{(k)}_e.
\end{align}
Here $\boldsymbol{A}_{\Omega^{(k)}_e}$ denotes the Boolean patch assembling matrix for $\Omega^{(k)}_e$. By means of the local Schur complements $\boldsymbol{S}_e^{(k)}$, see \equref{equ:loc_K} and \equref{equ:loc_Schur}, we can reformulate equation \equref{equ:hKu=hf} as
\begin{align}
 \label{equ:hSu=hg}
 \widehat{\boldsymbol{S}}_e u_{B_e} = \widehat{\boldsymbol{g}}_e,
\end{align}
where $\widehat{\boldsymbol{S}}_e$ and $\widehat{\boldsymbol{g}}_e$ are given by
  \begin{align}
  \label{equ:decomSchurMat}
    \widehat{\boldsymbol{S}}_e=\left(\sum_{k=1}^N \boldsymbol{A}_{\Gamma^{(k)}_e}\boldsymbol{S}^{(k)}_e\boldsymbol{A}^T_{\Gamma^{(k)}_e} \right) \text{ and}\quad 
    \widehat{\boldsymbol{g}}_e=\sum_{k=1}^N \boldsymbol{A}_{\Gamma^{(k)}_e}\boldsymbol{g}^{(k)}_e.
  \end{align}
  The Boolean matrix $\boldsymbol{A}_{\Gamma^{(k)}_e}$ is the corresponding assembling matrix and the vector $\boldsymbol{g}^{(k)}$ is defined by $\boldsymbol{g}^{(k)}= \boldsymbol{f}_{e,B_e}^{(k)} -  \boldsymbol{K}^{(k)}_{e,B_e I}\left(\boldsymbol{K}^{(k)}_{e,II}\right)^{-1}\boldsymbol{f}_{e,I}^{(k)}$. Furthermore, we can express \equref{equ:decomSchurMat} in operator notation as 
  \begin{align}
  \label{equ:decomSchurOp}
   \sum_{k=1}^N \langle S^{(k)}_e u_{B_e}^{(k)} , v^{(k)}\rangle = \sum_{k=1}^N \langle g_e^{(k)} , v^{(k)}\rangle \quad \forall v\in \widehat{W},
  \end{align}
 where $u_{B_e}\in \widehat{W},\, g_e^{(k)}\in\widehat{W}{^{(k)}}^*$ and $S^{(k)}_e: \widehat{W}^{(k)} \to \widehat{W}{^{(k)}}^*$. 
  
  In order to formulate the IETI-DP algorithm, we also define the Schur complement and the right-hand side functional on the ``discontinuous'' space $W$, i.e. 
     \begin{align*}
    S_e:W \to W^*, \quad\langle S_e v,w\rangle := \sum_{k=1}^N\langle S^{(k)}_e v^{(k)},w^{(k)}\rangle  \quad \forall v,w \in W,
    \end{align*}
    and
    \begin{align*}
    g_e\in W^*,\quad \langle g_e , w\rangle := \sum_{k=1}^N\langle g_e^{(k)},w^{(k)}\rangle  \quad \forall w \in W.	
   \end{align*}
In matrix form,
we can write $\boldsymbol{S}$ and $\boldsymbol{g}$ as
   \begin{align*}
       \boldsymbol{S}_e:= \text{diag}(\boldsymbol{S}^{(k)}_e)_{k=1}^N \text{ and} \quad \boldsymbol{g}_e:=[\boldsymbol{g}^{(k)}_e]_{k=1}^N.
   \end{align*}
   It is easy to see that problem \equref{equ:hSu=hg} is equivalent to
the minimization problem
    \begin{align}
      \label{equ:min_Schur}
    u_{B_e,h} = \underset{w\in W, Bw = 0}{\text{argmin}} \; \frac{1}{2} \langle S_e w , w\rangle -  \langle g_e,w\rangle.
   \end{align}
In the following, we will only work with the Schur complement system.
In order to simplify the notation, we will use $u$ instead of $u_{B,h}$, when we consider functions in $V_{\Gamma,h}$. 
If 
we have to made
a distinction between $u_h, u_{B,h}$ and $u_{I,h}$, 
we will add the subscripts again.

%

\subsection{Intermediate space and primal constraints}

   The key point of the dual-primal approach is the definition of an intermediate space $\widetilde{W}$ in the sense $\widehat{W}\subset\widetilde{W}\subset W$ such that $S_e$ restricted to $\widetilde{W}$ is positive definite.   Let $\Psi \subset V_{\Gamma,h}^*$ be a set of linearly independent \emph{primal variables}.
Then we define the spaces
\begin{equation*}
\widetilde{W} := \{w\in W: \forall\psi \in \Psi: \psi(w^{(k)}) = \psi(w^{(l)}), \forall k>l  \} 
\end{equation*}
and
\begin{equation*}
W_{\Delta} := \prod_{k=1}^N W_{\Delta}^{(k)},\text{ with} \quad W_{\Delta}^{(k)}:=\{w^{(k)}\in W^{(k)}: \forall\psi \in \Psi: \psi(w^{(k)}) =0 \}.
\end{equation*}
%
%
Moreover, we introduce the space $W_{\Pi} \subset \widehat{W}$ such that
$\widetilde{W} = W_{\Pi} \oplus W_{\Delta}.$
%
We call $W_{\Pi}$ \emph{primal space} and $W_{\Delta}$ \emph{dual space}. 
If we choose $\Psi$ such that $\widetilde{W} \cap \ker{S_e}=\{0\}$, then
  \begin{align*}
   \widetilde{S}_e: \widetilde{W} \to \widetilde{W}^*, \, \text{ with } \langle \widetilde{S}_e v,w\rangle =  \langle S_e v,w\rangle \quad \forall v,w \in \widetilde{W},
  \end{align*}
  is invertible. If a set $\Psi$ fulfils $\widetilde{W} \cap \ker{S_e} =\{0\}$, then we say that the set $\Psi$ \emph{controls the kernel}.
In the following, we will always assume that such a set is chosen. 
We will work with the following typical choices for the primal variables $\psi$:
\begin{itemize}
    \item Vertex evaluation: $\psi^\mathcal{V}(v) = v(\mathcal{V})$,
    \item Edge averages: $\psi^\mathcal{E}(v) = \frac{1}{|\mathcal{E}|}\int_{\mathcal{E}}v\,ds$,
    \item Face averages: $\psi^\mathcal{F}(v) = \frac{1}{|\mathcal{F}|}\int_{\mathcal{F}}v\,ds$.
\end{itemize}
Since we are considering a non-conforming test space $\widehat{W}$ and $V_{h}$,
we cannot literally use  the same set of primal variables as presented in \cite{HL:ToselliWidlund:2005a}, \cite{HL:Pechstein:2013a}, or \cite{HL:HoferLanger:2015a}. As proposed in \cite{HL:DryjaGalvisSarkis:2013a} and \cite{HL:DryjaSarkis:2014a}, we will use the following interpretation of continuity at corners, and continuous edge and face averages.
\begin{definition}
\label{def:primalVariables}
 Let $\mathcal{V}^{(k)}_e$, $\mathcal{E}^{(k)}$ and $\mathcal{F}^{(k)}$ be the set of vertices, edges and faces, respectively, for the patch $\Omega^{(k)}_e$, where in 2D the set $\mathcal{E}^{(k)}$ is empty. 
 
 We say that $u\in W$ is continuous at $\mathcal{V}^{(k)}$, $k\in\{1,\ldots,N\}$, if  
the relations
 \begin{align}
  (\boldsymbol{u}^{(k)})^{(k)}_i =  (\boldsymbol{u}^{(l)})^{(k)}_j \quad\forall (i,j)\in B_\mathcal{V}(k,l) 
 \end{align}
are valid for  all $l\in{\mathcal{I}}_{\mathcal{F}}^{(k)}$,
 where $B_\mathcal{V}(k,l)\subset B(k,l)$ is given by all index pairs corresponding to the vertices $\mathcal{V}_e^{(k)}$.
%
We define the corresponding primal variable as
\begin{align}
 \psi^{\nu^{(kl)}}
 (v):= \begin{cases}
                             (\boldsymbol{v}^{(k)})^{(k)}_i &\text{if }\;  v\in W^{(k)},\\
                             (\boldsymbol{v}^{(l)})^{(k)}_j &\text{if }\;  v\in W^{(l)},\\
                             0 & \text{else},
                            \end{cases}
\end{align}
where $l\in{\mathcal{I}}_{\mathcal{F}}^{(k)}$, $\nu^{(kl)}\in\mathcal{V}_e^{(kl)}$ and $(i,j)\in B_\mathcal{V}(k,l)$ corresponds to $\nu^{(kl)}$.
 
 We  say that $u\in W$ has continuous (inter-)face averages at $\mathcal{F}^{(k)}$, $k\in\{1,\ldots,N\}$, if  
the relations
\begin{align}
  \frac{1}{|F^{(kl)}|}\int_{F^{(kl)}}(u^{(k)})^{(k)}\,ds =  \frac{1}{|F^{(kl)}|}\int_{F^{(kl)}}(u^{(l)})^{(k)}\,ds    
\end{align}
hold for all $l\in{\mathcal{I}}_{\mathcal{F}}^{(k)}$.
 We define the corresponding primal variable as
\begin{align}
 \psi^{F^{(kl)}}(v):= \begin{cases}
                             \frac{1}{|F^{(kl)}|}\int_{F^{(kl)}}(u^{(k)})^{(k)}\,ds &\text{if }\;  v\in W^{(k)},\\
                             \frac{1}{|F^{(kl)}|}\int_{F^{(kl)}}(u^{(l)})^{(k)}\,ds &\text{if }\;  v\in W^{(l)},\\
                             0 & \text{else},
                            \end{cases}
\end{align}
where $l\in{\mathcal{I}}_{\mathcal{F}}^{(k)}$.

 We say that $u\in W$ has continuous edge averages at $\mathcal{E}^{(k)}$, $k\in\{1,\ldots,N\}$, if  
the relations
  \begin{align}
  \frac{1}{|E^{(klm)}|}\int_{E^{(klm)}}(u^{(k)})^{(k)}\,ds =  \frac{1}{|E^{(klm)}|}\int_{E^{(klm)}}(u^{(l)})^{(k)}\,ds,\\  
  \frac{1}{|E^{(klm)}|}\int_{E^{(klm)}}(u^{(k)})^{(k)}\,ds =  \frac{1}{|E^{(klm)}|}\int_{E^{(klm)}}(u^{(m)})^{(k)}\,ds   
 \end{align}
hold for all $(l,m)\in\mathcal{E}^{(k)}$.
 We define the corresponding primal variable as
\begin{align}
 \psi^{E^{(klm)}}(v):= \begin{cases}
                             \frac{1}{|E^{(klm)}|}\int_{E^{(klm)}}(u^{(k)})^{(k)}\,ds &\text{if }\;  v\in W^{(k)},\\
                             \frac{1}{|E^{(klm)}|}\int_{E^{(klm)}}(u^{(l)})^{(k)}\,ds &\text{if }\;  v\in W^{(l)},\\
                             \frac{1}{|E^{(klm)}|}\int_{E^{(klm)}}(u^{(m)})^{(k)}\,ds &\text{if }\;  v\in W^{(m)},\\
                             0 & \text{else},
                            \end{cases}
\end{align}
where $(l,m)\in\mathcal{E}^{(k)}$.
\end{definition}
By means of \defref{def:primalVariables}, we can now introduce different sets of primal variables $\Psi$: 
\begin{itemize}
    \item Algorithm A: $\Psi^A := \{\psi^\nu,\; \forall \nu\in\mathcal{V}^{(k)}\}_{k=1}^N$,
    \item Algorithm B: $\Psi^B := \{\psi^\nu,\; \forall \nu\in\mathcal{V}^{(k)}\}_{k=1}^N \cup \{\psi^E,\; \forall E\in\mathcal{E}^{(k)}\}_{k=1}^N \cup \{\psi^F,\; \forall F\in\mathcal{E}^{(k)}\}_{k=1}^N$,
    \item Algorithm C: $\Psi^C := \{\psi^\nu,\; \forall \nu\in\mathcal{V}^{(k)}\}_{k=1}^N \cup \{\psi^E,\; \forall E\in\mathcal{E}^{(k)}\}_{k=1}^N $.
\end{itemize}

%
%

\subsection{IETI - DP and preconditioning}
\label{sec:IETI_DP}
Since $\widetilde{W} \subset W$, there is a natural embedding $\widetilde{I}: \widetilde{W} \to W$. 
Let the jump operator restricted to $\widetilde{W}$ be
\begin{align}
   \label{def:Btilde}
    \widetilde{B} := B\widetilde{I} : \widetilde{W} \to U^*.
\end{align}
Then we can formulate problem \equref{equ:min_Schur} as saddle point problem in $\widetilde{W}$ as follows:
Find $(u,\boldsymbol{\lambda}) \in \widetilde{W} \times U:$
    \begin{align}
    \label{equ:saddlePointReg}
     \MatTwo{\widetilde{S}_e}{\widetilde{B}^T}{\widetilde{B}}{0} \VecTwo{u}{\boldsymbol{\lambda}} = \VecTwo{\widetilde{g}}{0},
    \end{align}
    where $\widetilde{g} := \widetilde{I}^T g$, and $\widetilde{B}^T= \widetilde{I}^T B^T$. Here, $\widetilde{I}^T: W^* \to \widetilde{W}^*$ denotes the adjoint of $\widetilde{I}$, which can be seen as a partial assembling operator.

By construction, $\widetilde{S}_e$ is SPD on $\widetilde{W}$. Hence, we can define the Schur complement $F$ 
and the corresponding right-hand side of equation \equref{equ:saddlePointReg} 
as follows:
\begin{align*}
    F:= \widetilde{B} \widetilde{S}_e^{-1}\widetilde{B}^T, \quad d:= \widetilde{B}\widetilde{S}_e^{-1} \widetilde{g}.
\end{align*}
Hence, the saddle point system \equref{equ:saddlePointReg} is equivalent to 
the Schur complement problem:
\begin{align}
   \label{equ:SchurFinal}
      \text{Find } \boldsymbol{\lambda} \in U: \quad F\boldsymbol{\lambda} = d.
\end{align}
Equation \equref{equ:SchurFinal} is solved by means of the PCG algorithm, but it requires an appropriate preconditioner in order to obtain an efficient solver. According to \cite{HL:DryjaGalvisSarkis:2013a} and \cite{HL:DryjaSarkis:2014a}, the right choice for FE is the \emph{scaled Dirichlet preconditioner}, adapted for the extended set of dofs. 
The numerical tests presented in Section~\ref{sec:numerical} 
indicate
that the scaled Dirichlet preconditioner works well  for the IgA setting too.

Recall the definition of $S_e = \text{diag}(S_e^{(k)})_{k=1}^N$, we define the scaled Dirichlet preconditioner $M_{sD}^{-1}$ as
   \begin{align}
   \label{equ:scaled_Dirichlet}
	M_{sD}^{-1} = B_D S_e B_D^T,
   \end{align}
where $B_D$ is a scaled version of the jump operator $B$. 
The scaled jump operator
$B_D$ is defined such that the operator enforces the constraints
 \begin{align}
 \label{equ:scaledB1}
  {\delta^\dagger}^{(l)}_j(\boldsymbol{u}^{(k)})^{(k)}_i -  {\delta^\dagger}^{(k)}_i(\boldsymbol{u}^{(l)})^{(k)}_j = 0 \quad\forall (i,j)\in B_e(k,l), \;\forall l\in{\mathcal{I}}_{\mathcal{F}}^{(k)},
 \end{align}
and
 \begin{align}
 \label{equ:scaledB2}
  {\delta^\dagger}^{(l)}_j(\boldsymbol{u}^{(k)})^{(l)}_i -  {\delta^\dagger}^{(k)}_i(\boldsymbol{u}^{(l)})^{(l)}_j = 0 \quad\forall (i,j)\in B_e(l,k), \;\forall l\in{\mathcal{I}}_{\mathcal{F}}^{(k)},
 \end{align}
 where for $(i,j)\in B_e(k,l)$
 \begin{align*}
        {\delta^\dagger}^{(k)}_i= \frac{\rho^{(k)}_i}{\sum_{l\in{\mathcal{I}}_{\mathcal{F}}^{(k)}} \rho^{(l)}_j} 
       \end{align*}
 is an appropriate scaling. Typical choices for $\rho^{(k)}_i$ are
       \begin{itemize}
     \item Multiplicity Scaling: $\rho^{(k)}_i = 1$,
     \item Coefficient Scaling: If $\alpha(x)_{|\Omega^{(k)}} = \alpha^{(k)}$, choose $\rho^{(k)}_i = \alpha^{(k)}$, 
     \item Stiffness Scaling: $\rho^{(k)}_i = {\boldsymbol{K}_e}_{i,i}^{(k)}$. 
    \end{itemize}
  If the diffusion coefficient $\alpha$ is constant and identical on each patch, then the multiplicity and the coefficient scaling are the same. If there is only a little variation in $\alpha$, then the multiplicity scaling provides good results. 
If the variation is 
really
large, then one should use the other scalings to obtain robustness 
with respect to the jumps in the diffusion coefficient across 
the patch interfaces.

  
 Since we can consider the dG-IETI-DP method as a conforming Galerkin (cG) method on an extended grid,
 we can implement the dG-IETI-DP algorithm following the implementation of the corresponding 
 cG-IETI-DP method given in \cite{HL:HoferLanger:2015a}.
 For completeness, we give an outline of the algorithm. The Schur complement system \equref{equ:SchurFinal} is solved using a CG algorithm with the preconditioner given in \equref{equ:scaled_Dirichlet}. The application of $F$ and $M_{sD}^{-1}$ is outlined in Algorithm~\ref{alg:applyF} and Algorithm~\ref{alg:applyMsD}. 
     \begin{algorithm}
   \caption{Algorithm for the calculation of $\boldsymbol{\nu} = F\boldsymbol{\lambda}$ for given $\boldsymbol{\lambda} \in U$}   
   \label{alg:applyF}
   \begin{algorithmic}
    \State Application of $B^T:$ $\{f^{(k)}\}_{k=1}^N = B^T\boldsymbol{\lambda}$
     \State Application of $\widetilde{I}^T:$ $\{\boldsymbol{f}_{\Pi},\{f_{\Delta}^{(k)}\}_{k=1}^N\} = \widetilde{I}^T\left(\{f^{(k)}\}_{k=1}^N\right)$
    \State  Application of $\widetilde{S}^{-1}_e:$
	  \begin{itemize}
	   \item  $\boldsymbol{w}_{\Pi} = \boldsymbol S_{e,\Pi\Pi}^{-1} \boldsymbol{f}_{\Pi}$
	   \item  $w_{\Delta}^{(k)} = {S_{e,\Delta\Delta}^{(k)}}^{-1}f_{\Delta}^{(k)} \quad \forall k=1,\ldots,N$ 
	  \end{itemize}
   \State  Application of $\widetilde{I}:$ $\{w^{(k)}\}_{k=1}^N = \widetilde{I}\left(\{\boldsymbol{w}_{\Pi},\{w_{\Delta}^{(k)}\}_{k=1}^N\} \right)$
   \State  Application of $B:$ $ \boldsymbol{\nu} = B\left( \{w^{(k)}\}_{k=1}^N \right)$
  \end{algorithmic}
  \end{algorithm}
       \begin{algorithm}
   \caption{Algorithm for the calculation of $\boldsymbol{\nu} = M_{sD}^{-1}\boldsymbol{\lambda}$ for given $\boldsymbol{\lambda} \in U$}
   \label{alg:applyMsD}
   \begin{algorithmic}
    \State Application of $B_D^T:$ $\{w^{(k)}\}_{k=1}^N = B_D^T\boldsymbol{\lambda}$
    \State  Application of $S_e:$
	  \begin{enumerate}
     \item Solve: $K^{(k)}_{e,II} x^{(k)} = -K^{(k)}_{e,IB}w^{(k)}$
     \item $v^{(k)} = K^{(k)}_{e,BB} w^{(k)} + K^{(k)}_{e,BI}x^{(k)}$.
	  \end{enumerate}
   \State  Application of $B_D:$ $ \boldsymbol{\nu} = B_D\left( \{v^{(k)}\}_{k=1}^N \right)$
  \end{algorithmic}
  \end{algorithm}

  In \cite{HL:DryjaGalvisSarkis:2013a} and \cite{HL:DryjaSarkis:2014a} it is proven for FE that the condition number behaves 
like the condition number of the preconditioned system
for the continuous FETI-DP method, see also \cite{HL:DryjaGalvisSarkis:2007a} for dG-BDDC FE preconditioners. From \cite{HL:HoferLanger:2015a} and \cite{HL:VeigaChoPavarinoScacchi:2013a}, we know that the condition number of the continuous IETI-DP and BDDC-IgA operators 
is also quasi-optimal with respect to the patch and mesh sizes. 
Therefore, we expect that the condition number of the   dG-IETI-DP operator behaves as
\begin{align*}
       \kappa(M_ {sD}^{-1} F_{\widetilde{U}}) \leq C \max_k\left(1+\log\left(\frac{H^{(k)}}{h^{(k)}}\right)\right)^2,
\end{align*}
where $H^{(k)}$ and $h^{(k)}$ are the patch size and mesh size, respectively, 
and the positive  constant $C$ is independent of 
$H^{(k)}$, $h^{(k)}$,
$h^{(k)}/h^{(l)}$,
and $\alpha$.
Our numerical results presented in the next section insistently confirm 
this behaviour.
  

\section{Numerical examples}
  \label{sec:numerical}
  
%
In this section, we present some numerical results documenting 
the numerical behaviour of the implemented dG-IETI-DP algorithm for solving large-scale linear systems arising from higher-order IgA discretizations
of \equref{equ:ModelStrong} in the domains illustrated in \figref{fig:YETI_footprint}(a) and \figref{fig:YETI_footprint}(b). 
The computational domain consists of 21 subdomains in both 2D and 3D. In both cases, one side of a patch boundary has inhomogeneous Dirichlet conditions, whereas all other sides have homogeneous Neumann conditions. 
We consider the case of non-matching meshes,
i.e. two neighbouring patches may have different mesh sizes $h^{(k)}$ and $h^{(l)}$. Due to our implementation of the dG formulation, we only consider nested meshes on the interface, i.e. the B-Spline spaces on the interfaces are nested. However, we note that the presented algorithm does not rely on this assumption.  Each subdomain has a diameter of $H^{(k)}$ and an associated mesh size of $h^{(k)}$. In the following, we use the abbreviation $H/h=\max_k H^{(k)}/h^{(k)}$. We consider B-Splines, 
where its degree is chosen as $p = 2$ and $p = 4$. In all numerical examples when increasing the degree from $2$ to $4$, we keep the smoothness of the space, i.e. increasing the multiplicity of the knots on the coarsest mesh. In order to solve the linear system \equref{equ:SchurFinal}, a PCG algorithm with the scaled Dirichlet preconditioner \equref{equ:scaled_Dirichlet} is performed. We use a zero initial guess, and  a reduction of the initial residual by a factor of $10^{-6}$ as stopping criterion. The numerical examples illustrate the dependence of the condition number of the IETI-DP preconditioned system on jumps in the diffusion coefficient $\alpha$, patch size $H$, mesh size $h$ and the degree $p$. In \secref{sec:dependence_hkhl}, we investigate the special case of increasing $h^{(k)}/h^{(l)}$ and its influence on the condition number. In all other tests we consider a fixed the ratio $h^{(k)}/h^{(l)}$.
  
  We use the C++ library G+Smo\footnote{\url{https://ricamsvn.ricam.oeaw.ac.at/trac/gismo/wiki/WikiStart}} for describing the geometry and performing the numerical tests,  see also \cite{HL:JuettlerLangerMantzaflarisMooreZulehner:2014a} and \cite{gismoweb}.
The boldface letters in Tables~\ref{table:hom3D}, \ref{table:jump3D} and \ref{table:hihj} mean that the problem size does not fit into our
Desktop PC with an Intel(R) Xeon(R) CPU E5-1650 v2 @ 3.50GHz and 16 GB main memory,
on which we performed the numerical experiments. The problems with boldface letters were calculated on a server with 8x Intel(R) Xeon(R) CPU E7-4870 2,4GHz and 1 TB main memory.
 \begin{figure}[h!]
  \begin{subfigmatrix}{3}
\subfigure[2D YETI-footprint]{\includegraphics[width=0.3\textwidth]{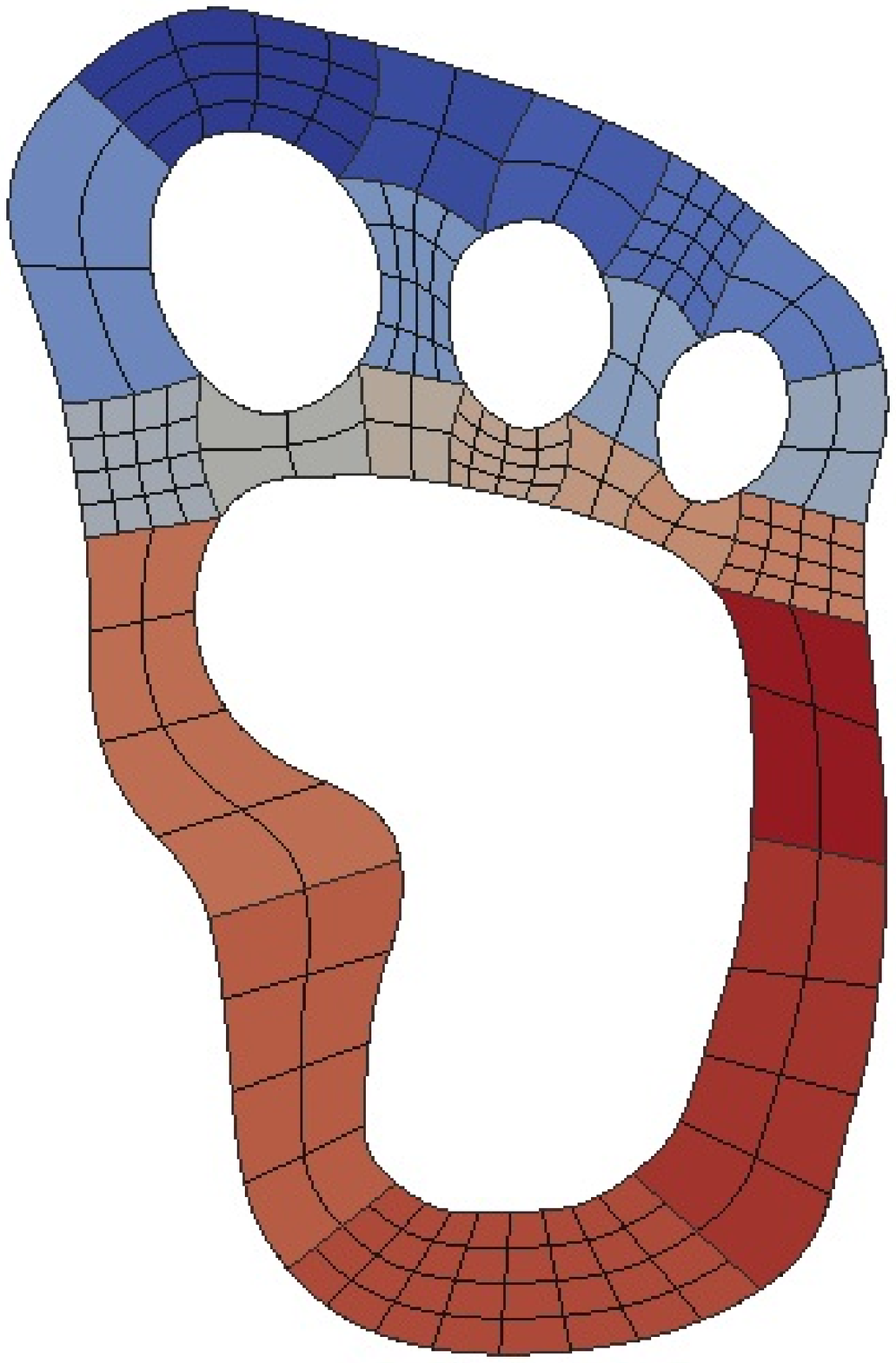} }
\subfigure[3D YETI-footprint]{\includegraphics[width=0.3\textwidth]{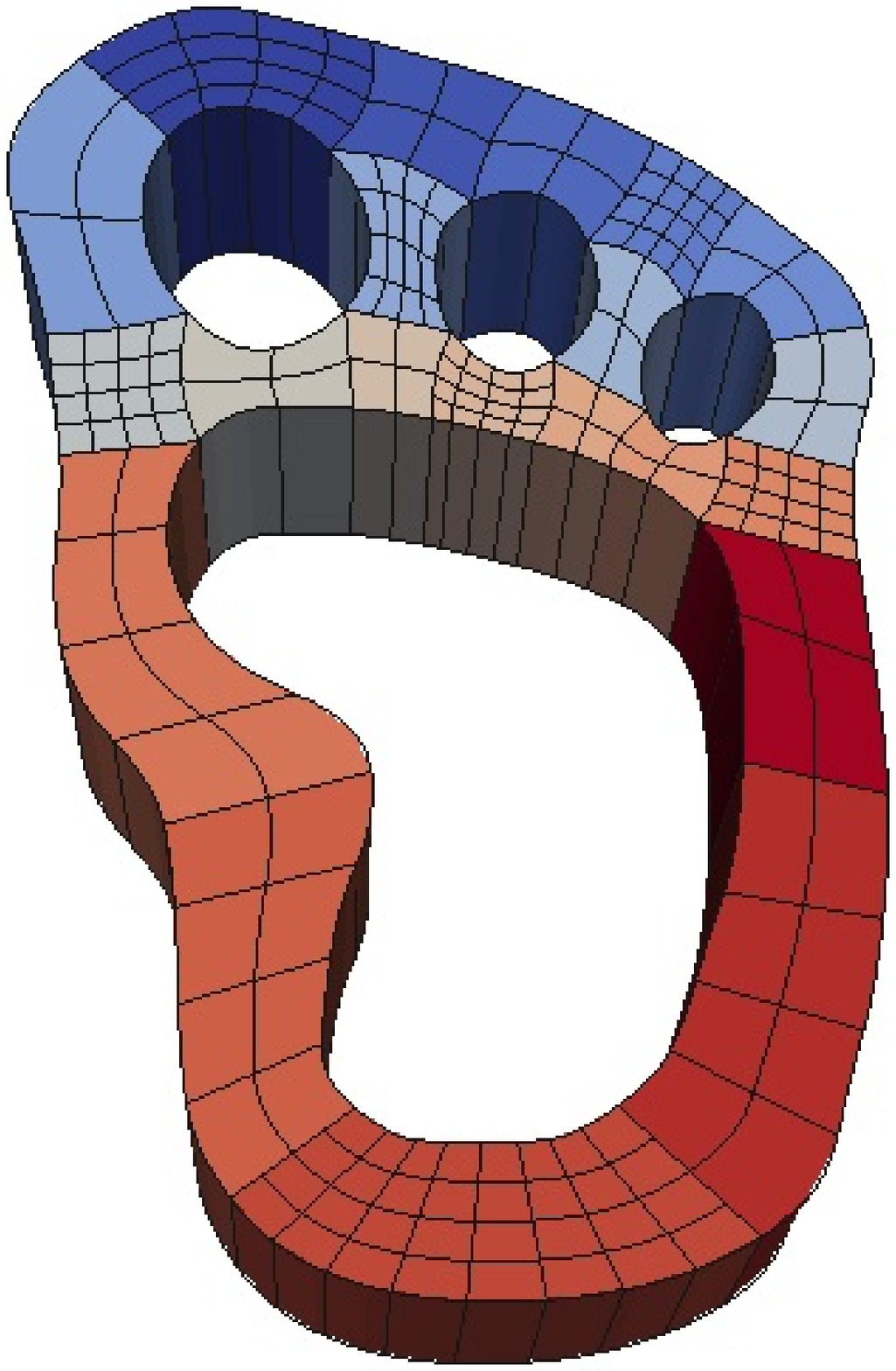} }
\subfigure[Pattern of jumping coefficients]{\includegraphics[width=0.3\textwidth]{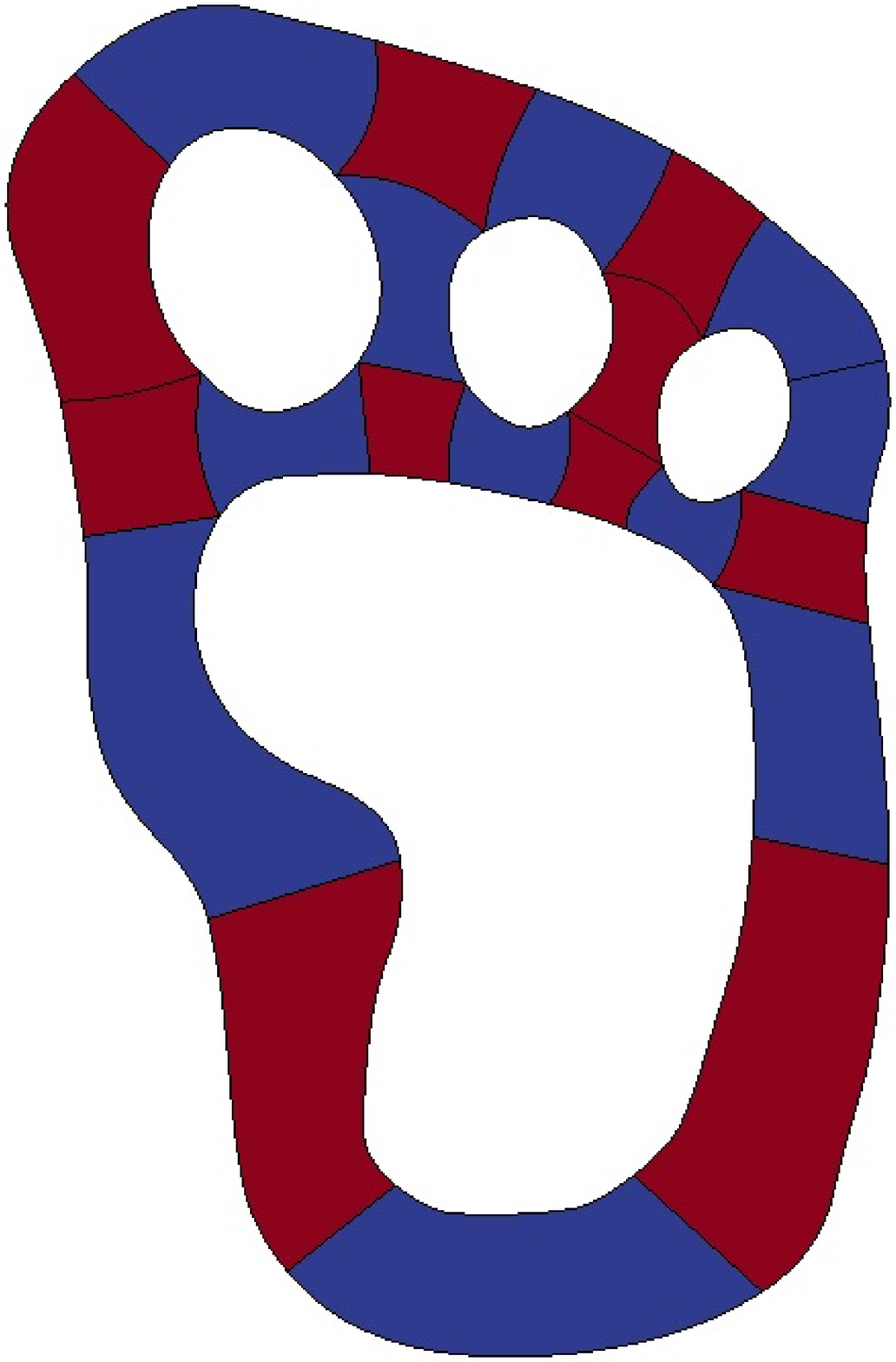} }
 \end{subfigmatrix}
  \caption{Figure (a) and (b) show 
  the computational domain, the decomposition into patches and its initial mesh in 2D and 3D, respectively. Figure (c) presents the pattern of the jumping diffusion coefficient. }
  \label{fig:YETI_footprint}
  \end{figure}
 \subsection{The case of homogeneous diffusion coefficient}
 \label{sec:hom_coeff}
 We first consider the case of a homogeneous diffusion coefficient, i.e. $\alpha=1$ on $\Omega$. The 2D results are summarized in \tabref{table:hom2D}, whereas the 3D results are presented in  \tabref{table:hom3D}. We observe, that the condition number of the preconditioned system grows logarithmically with respect to $H/h$. Moreover, the numerical results indicate a dependence of the condition number on the degree $p$, which will be investigated in more detail in \secref{sec:dependence_p}.
 
 \begin{table}[h!]
\centering
\begin{tabular}{|r|c|cc|cc||r|c|cc|cc|}\hline
\multicolumn{6}{|c||}{Degree $p=2$} 					             & \multicolumn{6}{|c|}{Degree $p=4$} \\\hline
 \multicolumn{2}{|c|}{ALG. A} &  \multicolumn{2}{|c|}{coeff. scal.} & \multicolumn{2}{|c||}{stiff. scal.}& \multicolumn{2}{|c|}{ALG. A} & \multicolumn{2}{|c|}{coeff. scal.} & \multicolumn{2}{|c|}{stiff. scal.}  \\ \hline
	    $\#$dofs & $H/h$ & $\kappa$ & It. & $\kappa$ & It. &$\#$dofs & $H/h$ & $\kappa$ & It. & $\kappa$ & It. \\ \hline
	      1610   & 8    &   3.07  &  17   &   3.16   &     17     &   4706  &   8   &   4.56  &  20  &  4.75  &   20   \\ \hline     
	      4706   & 16   &   4.06  &  19   &   4.22   &     19     &   9370  &  16   &   5.64  &  22  &  5.89  &   22    \\ \hline    
	      15602  & 32   &   5.22  &  21   &   5.45   &     21     &   23402 &  32   &   7.01  &  23  &  7.35  &   23    \\ \hline    
	      56210  & 64   &   6.55  &  23   &   6.86   &     23     &   70282 &  64   &   8.56  &  24  &  9.00  &   24    \\ \hline  
	      212690 & 128  &   8.04  &  24   &   8.45   &     24     &   239306&  128  &   10.3  &  26  &  10.8  &   26      \\  
	      \hline\hline
	  \multicolumn{2}{|c|}{ALG. C}  & \multicolumn{4}{|c||}{ } & \multicolumn{2}{|c|}{ALG. C}  & \multicolumn{4}{|c|}{ }\\ \hline
	      1610   & 8    &   1.35  &   9   &   1.34   &      9     &   4706  &   8   &   1.78  &  11  &  1.79  &   12   \\ \hline     
	      4706   & 16   &   1.64  &  11   &   1.64   &     11     &   9370  &  16   &   2.11  &  13  &  2.11  &   13    \\ \hline    
	      15602  & 32   &   1.99  &  12   &   1.99   &     12     &   23402 &  32   &   2.54  &  15  &  2.53  &   15    \\ \hline    
	      56210  & 64   &   2.41  &  14   &   2.41   &     14     &   70282 &  64   &   3.03  &  16  &  3.00  &   16    \\ \hline  
	      212690 & 128  &   2.88  &  16   &   2.88   &     16     &   239306&  128  &   3.57  &  17  &  3.54  &   18      \\ \hline
  \end{tabular}
  \caption{2D example with $p = 2$ (left) and $p = 4$ (right), and homogeneous diffusion coefficient. Dependence of the condition number $\kappa$ and the number It. of iterations on $H/h$ for the preconditioned system with coefficient and stiffness scaling. 
  Choice of primal variables: vertex evaluation (upper table), vertex evaluation and edge averages (lower table).}
\label{table:hom2D}
  \end{table}
\begin{table}[h!]
\centering
\begin{tabular}{|r|c|cc|cc||r|c|cc|cc|}\hline
\multicolumn{6}{|c||}{Degree $p=2$} 					             & \multicolumn{6}{|c|}{Degree $p=4$} \\ \hline
 \multicolumn{2}{|c|}{ALG. A} &  \multicolumn{2}{|c|}{coeff. scal.} & \multicolumn{2}{|c||}{stiff. scal.}& \multicolumn{2}{|c|}{ALG. A} & \multicolumn{2}{|c|}{coeff. scal.} & \multicolumn{2}{|c|}{stiff. scal.}  \\ \hline
	    $\#$dofs & $H/h$ & $\kappa$ & It. & $\kappa$ & It. &$\#$dofs & $H/h$ & $\kappa$ & It. & $\kappa$ & It. \\ \hline
	      2800   & 3    &   49.9  &  44   &   49.3   &     44     &   22204  &   3   &   203  &  77  &  199  &   79   \\ \hline     
	      9478   & 6    &   72.3  &  48   &   70.3   &     49     &   42922  &   6   &   248   &  82  &  240   &   83    \\ \hline    
	      42922  & 12   &   169   &  70   &   165    &     69     &   116110 &   12  &   506  &  104  &  488  &   104    \\ \hline    
	      244594 & 25   &   376   &  91   &   368    &     92     &  \bf 443926  & \bf 25   & \bf 1090  & \bf 130  & \bf 1060 & \bf 129  \\ 	      
	      \hline\hline
	  \multicolumn{2}{|c|}{ALG. B}  & \multicolumn{4}{|c||}{ } & \multicolumn{2}{|c|}{ALG. B}  & \multicolumn{4}{|c|}{ }\\ \hline
	      2800   & 3    &   1.49  &  9    &   1.45   &     9      &   22204  &   3   &   23.3  &  33  &  21.5  &   34   \\ \hline     
	      9478   & 6    &   15.8  &  17   &   14.7   &     17     &   42922  &   6   &   26.7  &  34  &  24.7   &   34    \\ \hline    
	      42922  & 12   &   19.8  &  30   &   18.4   &     29     &   116110 &   12  &   31.3  &  42  &  29.1  &   41    \\ \hline    
	      244594 & 25   &   24   &  37   &   22.4    &     36     &   \bf 443926 &  \bf 25 & \bf 36.4 & \bf44  & \bf34.1 & \bf45   \\ \hline  
  \end{tabular}
  \caption{3D example with $p = 2$ (left) and $p = 4$ (right), and homogeneous diffusion coefficient. Dependence of the condition number $\kappa$ and the number It. of iterations on $H/h$ for the preconditioned system with coefficient and stiffness scaling. Choice of primal variables: vertex evaluation (upper table), vertex evaluation and edge averages and face averages (lower table).}
\label{table:hom3D}
  \end{table}
  
 \subsection{The case of inhomogeneous diffusion coefficient}
 \label{sec:jumping_coeff}
 In this subsection, we investigate the case of patchwise constant diffusion coefficient, but with jumps across the patch interfaces. The diffusion coefficient takes values $\alpha^{(k)}\in \{10^{-4}, 10^4\}$, with a jumping pattern according to \figref{fig:YETI_footprint} (c). The 2D results are summarized in \tabref{table:jump2D}, and the 3D results are presented in  \tabref{table:jump3D}. First of all, one clearly sees the robustness with respect to jumping coefficients of the considered method and the quasi optimal dependence of the condition number on $H/h$. The dependence of the degree will again be studied in \secref{sec:dependence_p}.
   %
  \begin{table}[h!]
\centering
\begin{tabular}{|r|c|cc|cc||r|c|cc|cc|}\hline
\multicolumn{6}{|c||}{Degree $p=2$} 					             & \multicolumn{6}{|c|}{Degree $p=4$} \\\hline
 \multicolumn{2}{|c|}{ALG. A} &  \multicolumn{2}{|c|}{coeff. scal.} & \multicolumn{2}{|c||}{stiff. scal.}& \multicolumn{2}{|c|}{ALG. A} & \multicolumn{2}{|c|}{coeff. scal.} & \multicolumn{2}{|c|}{stiff. scal.}  \\ \hline
	    $\#$dofs & $H/h$ & $\kappa$ & It. & $\kappa$ & It. &$\#$dofs & $H/h$ & $\kappa$ & It. & $\kappa$ & It. \\ \hline
	      1610   & 8    &   3.82  &  12   &   4.02   &     12     &   4706  &   8   &   5.72  &  14  &  6.16  &   14   \\ \hline     
	      4706   & 16   &   5.11  &  13   &   5.47   &     13     &   9370  &  16   &   7.08  &  14  &  7.7   &   15    \\ \hline    
	      15602  & 32   &   6.58  &  15   &   7.12   &     15     &   23402 &  32   &   8.77  &  15  &  9.64  &   17    \\ \hline    
	      56210  & 64   &   8.23  &  15   &   9      &     16     &   70282 &  64   &   10.7  &  18  &  11.8  &   18    \\ \hline  
	      212690 & 128  &   10.1  &  17   &   11.1   &     18     &   239306&  128  &   12.8  &  18  &  14.2  &   18      \\  
	      \hline\hline
	  \multicolumn{2}{|c|}{ALG. C}  & \multicolumn{4}{|c||}{ } & \multicolumn{2}{|c|}{ALG. C}  & \multicolumn{4}{|c|}{ }\\ \hline
	      1610   & 8    &   1.4  &   7   &   1.43   &      7     &   4706  &   8   &   1.85  &  8  &  1.94  &   8   \\ \hline     
	      4706   & 16   &   1.7  &   7   &   1.78   &      7     &   9370  &  16   &   2.17  &  8  &  2.32  &   8    \\ \hline    
	      15602  & 32   &   2.06 &   8   &   2.19   &      8     &   23402 &  32   &   2.58  &  9  &  2.81  &   9    \\ \hline    
	      56210  & 64   &   2.46 &   8   &   2.65   &      8     &   70282 &  64   &   3.05  &  9  &  3.36  &   9   \\ \hline  
	      212690 & 128  &   2.9  &   9   &   3.18   &      9     &   239306&  128  &   3.55  &  10 &  3.97  &   10     \\ \hline
  \end{tabular}
  \caption{2D example with $p = 2$ (left) and $p = 4$ (right), and 
jumping diffusion coefficient. Dependence of the condition number $\kappa$ and the number It. of iterations on $H/h$ for the preconditioned system with coefficient and stiffness scaling. 
  Choice of primal variables: vertex evaluation (upper table), vertex evaluation and edge averages (lower table).}
\label{table:jump2D}
  \end{table}
\begin{table}[h!]
\centering
\begin{tabular}{|r|c|cc|cc||r|c|cc|cc|}\hline
\multicolumn{6}{|c||}{Degree $p=2$} 					             & \multicolumn{6}{|c|}{Degree $p=4$} \\ \hline
 \multicolumn{2}{|c|}{ALG. A} &  \multicolumn{2}{|c|}{coeff. scal.} & \multicolumn{2}{|c||}{stiff. scal.}& \multicolumn{2}{|c|}{ALG. A} & \multicolumn{2}{|c|}{coeff. scal.} & \multicolumn{2}{|c|}{stiff. scal.}  \\ \hline
	    $\#$dofs & $H/h$ & $\kappa$ & It. & $\kappa$ & It. &$\#$dofs & $H/h$ & $\kappa$ & It. & $\kappa$ & It. \\ \hline
	      2800   & 3    &   50.3  &  28   &   57.9   &     25     &   22204  &   3   &   203  &  44  &  236  &   45   \\ \hline     
	      9478   & 6    &   72.2  &  29   &   83.4   &     29     &   42922  &   6   &   250  &  43  &  290  &   42    \\ \hline    
	      42922  & 12   &   176   &  43   &   203    &     42     &   116110 &   12  &   520  &  58  &  605  &   57    \\ \hline    
	      244594 & 25   &   400   &  52   &   463    &     58     &  \bf 443926 &  \bf 25   & \bf 1143 &  \bf 72 & \bf 1331  &  \bf 80   \\ 	      
	      \hline\hline
	  \multicolumn{2}{|c|}{ALG. B}  & \multicolumn{4}{|c||}{ } & \multicolumn{2}{|c|}{ALG. B}  & \multicolumn{4}{|c|}{ }\\ \hline
	      2800   & 3    &   2.11  &  11   &   2.17   &     11     &   22204  &   3   &   17.7  &  15  &  20.7  &   15   \\ \hline     
	      9478   & 6    &   12.6  &  17   &   14.6   &     18     &   42922  &   6   &   20.5  &  17  &  23.9  &   17    \\ \hline    
	      42922  & 12   &   15.7  &  22   &   18.2   &     24     &   116110 &   12  &   24  &  19  &  28 &   21    \\ \hline    
	      244594 & 25   &   18.9  &  28   &   22     &     30     &   \bf 443926  & \bf 25   & \bf 28  & \bf 22  & \bf 32.6 & \bf 22   \\ \hline  
  \end{tabular}
  \caption{3D example with $p = 2$ (left) and $p = 4$ (right), and jumping diffusion coefficient. Dependence of the condition number $\kappa$ and the number It. of iterations on $H/h$ for the preconditioned system with coefficient and stiffness scaling. Choice of primal variables: vertex evaluation (upper table), vertex evaluation and edge averages and face averages (lower table).}
\label{table:jump3D}
  \end{table}


 \subsection{Dependence in $h^{(k)}/h^{(l)}$}
 \label{sec:dependence_hkhl}
 In this subsection, we deal with dependence of the condition number  on the ratio $q=h^{(k)}/h^{(l)}$ of mesh sizes corresponding to neighbouring patches,
The initial domain is the 
same 
as given in \figref{fig:YETI_footprint}, but without the additional refinements in 
certain patches, i.e. $h^{(k)}/h^{(l)}=1$. Then we consequently perform uniform refinement in the considered patches and obtain $h^{(k)}/h^{(l)}= 2^{-r}$, where $r$ is the number of refinements. 
In the numerical tests, we only consider the cases $\alpha\in\{10^{-4}, 10^4\}$ and $p = 4$.
The results for 2D and 3D are summarized in \tabref{table:hihj} and indicate that the condition number is independent of the ratio $h^{(k)}/h^{(l)}$ for 2D and 3D, as also predicted for FE in \cite{HL:DryjaGalvisSarkis:2013a}. We note that the increasing condition number and number of iterations come along with the increased ratio $H/h$ and comparing the numbers of this test with the corresponding ones from \tabref{table:jump3D} we observe an agreement. Thus, it is 
noteworthy that, although in 2D the ratio $H/h$ is increasing, the condition number stays constant.

 \begin{table}[h!]
\centering
\begin{tabular}{|r|c|c|cc|cc||r|c|c|cc|cc|}\hline
\multicolumn{7}{|c||}{ dim = 2} 					             & \multicolumn{7}{|c|}{dim = 3} \\ \hline
 \multicolumn{3}{|c|}{ALG. A} &  \multicolumn{2}{|c|}{coeff.} & \multicolumn{2}{|c||}{stiff.}& \multicolumn{3}{|c|}{ALG. C} & \multicolumn{2}{|c|}{coeff.} & \multicolumn{2}{|c|}{stiff.}  \\ \hline
	    $\#$dofs &$q$ &$H/h$&$\kappa$&It.&$\kappa$&It.&$\#$dofs&$q$ &$H/h$&$\kappa$&It.&$\kappa$& It.\\ \hline
	      1816    & 1    & 2 &  4.92  &  13   &  5.21   &     14     &   9362  &   1 & 1  &   14.8  &  21  & 14.9  &   21   \\ \hline     
	      2134   & 2    & 4 &  4.93  &  13   &   5.36   &     14     &   11902 &   2 & 3  &   17.7  &  23  & 22.1  &   24    \\ \hline    
	      2962   & 4    & 8 &  4.93  &  13   &   5.55   &     14     &   20426 &   4 & 6  &  29.2  &  24  & 37.5  &   27    \\ \hline    
	      5386   & 8    & 16 & 4.93  &  13   &   5.69   &     14     &   56626 &   8 & 12 & 52.2  &  26  & 67.2  &   27   \\ \hline     
	      13306  & 16   & 32 & 4.93  &  13   &   5.71   &     14     & \bf 345268  &  \bf 16 &\bf  25 & \bf 98.4  & \bf 27  & \bf 124 & \bf  28  \\ \hline    
	      41434  & 32   & 64 & 4.92  &  13   &   5.66   &     14     & \bf 1758004 &  \bf 32 &\bf  50 & \bf 191  & \bf 28  & \bf 240 & \bf  28   \\ 	      
	      \hline\hline
	  \multicolumn{3}{|c|}{ALG. C}  & \multicolumn{4}{|c||}{ } & \multicolumn{3}{|c|}{ALG. B}  & \multicolumn{4}{|c|}{ }\\ \hline
	      1816    & 1    & 2 & 1.67  &  7    &   1.72   &     7      &   9362   &   1 &  1  &   14.8  &  15  & 16.6  &   15   \\ \hline     
	      2134   & 2    & 4 & 1.67 &  7      &   1.77   &     7      &   11902  &   2 &  3  &   19.5  &  15  & 28.7  &   17    \\ \hline    
	      2962   & 4    & 8 & 1.67  &  7     &   1.81   &     7      &   20426  &   4 &  6  &    29.2  &  16  & 37.5  &   18    \\ \hline    
	      5386   & 8    & 16&  1.67  &  7    &   1.85   &     7      &   56626  &   8 & 12  &   52.2  &  17  & 67.2  &   18   \\ \hline     
	      13306  & 16   & 32&  1.67  &  7    &   1.85   &     7      &\bf  345268   &  \bf  16& \bf 25  & \bf 98.4 & \bf 17 & \bf 124 &  \bf 19  \\ \hline    
	      41434  & 32   & 64&  1.67  &  7    &   1.84   &     7      &\bf  1758004  &  \bf  32& \bf 50  & \bf 308  & \bf 20  & \bf 240 &  \bf 20 \\ \hline
  \end{tabular}
  \caption{2D (left) and (3D)example with $p = 4$, and jumping diffusion coefficient. Dependence of the condition number $\kappa$ and the number It. of iterations on the ratio $q=h^{(k)}/h^{(l)}$ for the preconditioned system with coefficient and stiffness scaling. Choice of primal variables: in 2D vertex evaluation (upper table), vertex evaluation and edge averages (lower table), in 3D vertex evaluation and edge averages (upper table), vertex evaluation, edge averages and face averages (lower table).}
\label{table:hihj}
  \end{table}

 
\subsection{Dependence on $p$}
 \label{sec:dependence_p}
 In this subsection,  we study the dependence of the condition number on the degree $p$ of the B-Spline space. There are two ways to dealing with degree elevation. One method is to keep the smoothness of the space, i.e., the multiplicity of the knots is increased in each step. The other way keeps the multiplicity of the knots, while increasing the smoothness of the B-Spline. The first method retains the support of the B-Spline basis small, with the drawback of a larger number of dofs, while the second method does it vice versa, i.e., increasing the support of the B-Spline, while having a smaller number of dofs. The aim of this section is to investigate the effect of the two different elevation techniques on the condition number.
 
 We choose the computational domain  as the 2D and 3D YETI-footprint presented in \figref{fig:YETI_footprint} and the diffusion coefficient is chosen to be globally constant. The results are summarized in \tabref{table:pDependence2D} and in \tabref{table:pDependence3D} for the 2D and 3D domain, respectively. The numerical results indicate a at most linear dependence of the condition number of the preconditioned system on the B-Spline degree $p$. When considering the 2D domain, the dependence on the degree 
seems to
be 
even
logarithmic, see \figref{fig:pDependence}. One observes a significant increase of the condition number, when increasing the degree from 2 to 3 in 3D as illustrated in \figref{fig:pDependence}~(b). 
 
 %
\begin{table}[h!]
\centering
\begin{tabular}{|r| c|cc|cc||r| c|cc|cc|}\hline
      \multicolumn{6}{|c||}{Increasing the multiplicity} & \multicolumn{6}{|c|}{Increasing the smoothness} \\ \hline
      \multicolumn{2}{|c|}{ALG. C} &  \multicolumn{2}{|c|}{coeff. scal} & \multicolumn{2}{|c||}{stiff. scal.}&\multicolumn{2}{|c|}{ALG. C} &  \multicolumn{2}{|c|}{coeff. scal} & \multicolumn{2}{|c|}{stiff. scal.} \\ \hline
	      $\#$dofs & degree & $\kappa$ & It. & $\kappa$ & It. & $\#$dofs & degree & $\kappa$ & It. & $\kappa$ & It. \\ \hline
1610  &  2    &    1.36  &  9    &    1.37  &   9     &   1610 & 2  & 1.36   &  9    &  1.37   &  9   \\ \hline 
4706  &  3    &    1.68  &  11   &    1.69  &   11    &   2006 & 3 &  1.55   &  10   &  1.57   &  11  \\ \hline 
9370  &  4    &    1.95  &  12   &    1.96  &   13    &   2444 & 4 &  1.74   &  11   &  1.77   &  12 \\ \hline 
15602 &  5    &    2.19  &  13   &    2.2   &   14    &   2924 & 5 &  1.88   &  12   &  1.93   &  12 \\ \hline 
23402 &  6    &    2.4   &  15   &    2.4   &   15    &   3446 & 6 &  2.03   &  13   &  2.09   &  13 \\ \hline 
32770 &  7    &    2.59  &  15   &    2.59  &   16    &   4010 & 7 &  2.14   &  14   &  2.22   &  14 \\ \hline 
43706 &  8    &    2.77  &  16   &    2.76  &   16    &   4616 & 8 &  2.27   &  14   &  2.36   &  14 \\ \hline 
56210 &  9    &    2.93  &  17   &    2.92  &   17    &   5264 & 9 &  2.36   &  15   &  2.47   &  15 \\ \hline 
70282 &  10   &    3.08  &  17   &    3.06  &   17    &   5954 & 10 &  2.48   &  15   &  2.59   &  15 \\ \hline 
  \end{tabular}
   \caption{2D example with fixed initial mesh and homogeneous diffusion coefficient. Dependence of the condition number $\kappa$ and the number It. of iterations on $H/h$ for the preconditioned system with coefficient and stiffness scaling. Choice of primal variables: vertex evaluation and edge averages.}
\label{table:pDependence2D}
  \end{table}
\begin{table}[h!]
\centering
\begin{tabular}{|r| c|cc|cc||r| c|cc|cc|}\hline
      \multicolumn{6}{|c||}{Increasing the multiplicity} & \multicolumn{6}{|c|}{Increasing the smoothness} \\ \hline
      \multicolumn{2}{|c|}{ALG. B} &  \multicolumn{2}{|c|}{coeff. scal} & \multicolumn{2}{|c||}{stiff. scal.}&\multicolumn{2}{|c|}{ALG. B} &  \multicolumn{2}{|c|}{coeff. scal} & \multicolumn{2}{|c|}{stiff. scal.} \\ \hline
	      $\#$dofs & degree & $\kappa$ & It. & $\kappa$ & It. & $\#$dofs & degree & $\kappa$ & It. & $\kappa$ & It. \\ \hline
2800   &  2    &    1.49  &  9    &    1.45  &   9     &   2800  & 2  & 1.49   &  9    &  1.45   &  9   \\ \hline 
9478   &  3    &    17.9  &  20   &    16.6  &   20    &   4864  & 3 &  17.1   &  18   &  16.4   &  18  \\ \hline 
22204  &  4    &    23.3  &  33   &    21.5  &   34    &   7714  & 4 &  21.7   &  32   &  21   &    33 \\ \hline 
42922  &  5    &    29.0  &  39   &    26.9  &   39    &   11476 & 5 &  27.0   &  45   &  26.3   &  46 \\ \hline 
73576  &  6    &    34.4  &  51   &    32.0  &   49    &   16276 & 6 &  31.9   &  47   &  31.3   &  47 \\ \hline 
116110 &  7    &    40.2  &  54   &    37.6  &   55    &   22240 & 7 &  37.5   &  51   &  36.9   &  50 \\ \hline 
  \end{tabular}
   \caption{3D example with fixed initial mesh and homogeneous diffusion coefficient. Dependence of the condition number $\kappa$ and the number It. of iterations on $H/h$ for the preconditioned system with coefficient and stiffness scaling. Choice of primal variables: vertex evaluation, edge averages and face averages.}
\label{table:pDependence3D}
  \end{table}
   \begin{figure}[h!]
  \begin{subfigmatrix}{2}
\subfigure[2D setting]{\includegraphics[width=0.45\textwidth]{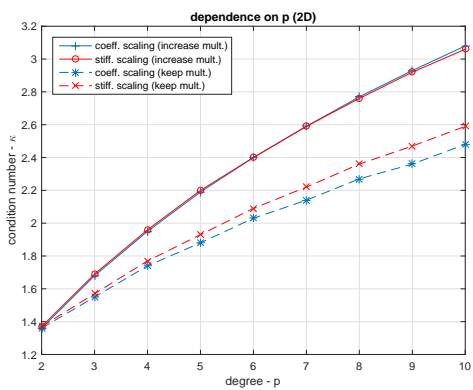} }
\subfigure[3D setting]{\includegraphics[width=0.45\textwidth]{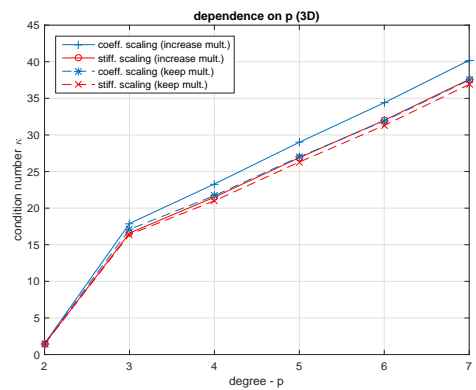} }
 \end{subfigmatrix}
  \caption{Dependence of the condition number on the B-Spline degree $p$ for the 2D and 3D domain. We compare the influence of the considered scaling strategy and method for increasing the degree.}
  \label{fig:pDependence}
  \end{figure}
%


 \subsection{Performance and OpenMP}
As already explained in the beginning of this section, we use the open source  C++ library G+SMO for materialising the code. The 
LU-factorizations for the local solvers were performed by means of the PARDISO 5.0.0 Solver Project, see \cite{HL:PARDISO500}. 
In \tabref{table:timings}, we investigate the runtime of a serial implementation and compare it with the timings of the cG-IETI-DP algorithm as presented in \cite{HL:HoferLanger:2015a}. In order to compare the timings for the continuous IETI-DP and discontinuous IETI-DP method, we use the setting as described in \cite{HL:HoferLanger:2015a}. 
More precisely,
the computational domain is the 2D example from \figref{fig:YETI_footprint}, but with fully matching patches, i.e. no additional refinements on selected patches. 
We observe from \tabref{table:timings} that the dG-IETI-DP method 
shows
a very similar performance. The increased number of primal 
variables and the extended version of the stiffness matrix $\boldsymbol{K}_e$ leads to a slightly larger runtime. As already mentioned, we use the PARDISO Solver Project instead of the SparseLU factorization of the open source library ``Eigen''\footnote{\url{http://eigen.tuxfamily.org/index.php?title=Main_Page}} as in Subsection 5.4 in \cite{HL:HoferLanger:2015a}. This change led to a significant speed-up in the computation time for the LU-factorization. Small deviations in the timings compared to Table 6 in \cite{HL:HoferLanger:2015a} might be due to some changes regarding the parallel implementation. 
 
   \begin{table}[h!]
   \centering
   \begin{tabular}{|l|r|r||r|r|}
   \hline
      & \multicolumn{2}{|c||}{Wall-clock time} & \multicolumn{2}{|c|}{relative time in \%} \\ \hline
      &\,cG\,&\,dG\,&\,cG\,&\,dG\,\\ \hline\hline
    Preparing the bookkeeping		 				& 0.012 s &  0.02 s & 0.06   &   0.11    \\ \hline\hline
    Assembling all patch local $\boldsymbol{K}^{(k)}$ 				&  6.4 s  &   6.8 s    & 35.56  & 35.79    \\
    Partitioning w.r.t. $B$ and $I$ 	 				& 0.085 s &   0.12 s  & 0.48  &  0.63  \\
    Assembling C 		  	 				& 0.017 s &  0.034 s & 0.09  &   0.18   \\
    LU-fact. of $\boldsymbol{K}_{II}^{(k)}$	 				& 2.5 s &     2.5 s   & 13.89  &  13.16   \\
    LU-fact. of $\MatTwo{K^{(k)}}{{C^{(k)}}^T}{C^{(k)}}{0}$	& 3.9 s   &  4 s   & 21.67  &   21.05     \\
    Assembling and LU-fact. of $\boldsymbol{S}_{\Pi\Pi}$ 					& 0.78 s  &  1.1 s   & 4.33   & 5.79     \\
    Assemble rhs. 							& 0.13 s &   0.13 s  &0.72    & 0.68    \\ \hline
    Total assembling 							& 14 s    &  15 s   & 77.78   &   78.95 \\ \hline \hline
    One PCG iteration 							& 0.34 s  & 0.36 s &  -   &   -   \\ \hline
    Solving the system 							& 3.4 s   &   3.6 s   & 18.89   & 18.95   \\ \hline \hline
    Calculating the solution $\boldsymbol{u}$ 					& 0.33 s  &   0.33 s  & 1.83   & 1.74    \\ \hline\hline
    Total spent time							& 18 s  &       19 s  &100.00  & 100.00 \\ \hline
   \end{tabular}
   \caption{Serial computation times of the 2D example  with coefficient scaling and Algorithm C. The discrete problem consists of $121824$ total degrees of freedom, $1692$ Lagrange multipliers, and on each patch approximate $4900$ local degrees of freedom according to the setting in \cite{HL:HoferLanger:2015a}, Section~5. Column 2 and 3 present the absolute spent time, whereas column 4 and 5 present the relative one for the cG-IETI-DP and dG-IETI-DP method.}
\label{table:timings}
  \end{table}
 The IETI-DP method is well suited for a parallel implementation, since most of the computations are independent of each other. Only the assembly and application of the Schur complement $\boldsymbol{S}_{\Pi\Pi}$ corresponding to the primal variables requires communication with  neighbouring patches. 
Although the structure of the algorithm 
perfectly
suits the 
framework of distributed memory models, hence, using MPI, we first implemented a version using OpenMP. We want to mention that the numerical examples in \secref{sec:hom_coeff} and \secref{sec:jumping_coeff} are performed by means of the parallel implementation. The major problem for obtaining a scaleable method is the unequally distributed workload. This arises from the fact that in IgA the partition of the domain is mostly based on geometric aspect, where as in FE the mesh is partitioned in such a way, that each patch has a similar number of dofs. Especially, the cases considered in \secref{sec:hom_coeff} and \secref{sec:jumping_coeff}, where we have non-matching meshes, lead to 
very unequally distributed workloads. To summarize, in order to achieve a scalable IETI type solver one has to spend some time in finding an equally distributed workload for each thread, e.g. performing further subdivisions of certain patches and optimal assignment of patches to threads.
Alternatively, one can use inexact IETI-DP versions in which the subdomain problems are 
solved respectively preconditioned by robust parallel multigrid methods 
like proposed in \cite{HL:HofreitherTakacsZulehner:2015a}.

\section{Conclusions and outlook}
\label{HL:Sec:Conclusions}
  In this paper, we investigated an adaption of the IETI-DP method 
  to
  dG-IgA equations, i.e. we used dG techniques to couple non-matching meshes across patch interfaces. The numerical examples in \secref{sec:numerical} indicate the same quasi-optimal behaviour of the condition number of the dG-IETI-DP operator with respect to $H/h$, and  show robustness with respect to jumps in the diffusion coefficient. Additionally, the condition number in 2D and 3D seems to be independent of the ratio $h^{(k)}/h^{(l)}$. 
Moreover, we examined the dependence of the condition number on the B-Spline degree. We found that in 2D the dependence is quite weak, while in 3D one observes a more significant increase. As illustrated in \figref{fig:pDependence}, the dependence seems to be  linear or even logarithmic in 2D, while it is clearly linear in 3D. Finally, we investigated the performance of the dG-IETI-DP 
in comparison with
the cG-IETI-DP method, which turns out to be very similar. 
The theoretical analysis, 
that was done for cG-IETI-DP in 
\cite{HL:HoferLanger:2015a}, is more technical for dG-IETI-DP,
cf. \cite{HL:DryjaGalvisSarkis:2013a} for dG-FETI-DP.
The dG-IETI-DP can be generalized to dG-IgA schemes 
on segmentations with non-matching interfaces (segmentation crimes) 
studied in
\cite{HL:HoferLangerToulopoulos_2015a}
and 
\cite{HL:HoferToulopoulos_2015a}.

\section*{Acknowledgment}
The research is supported by the Austrian Science Fund (FWF) through the
NFN S117-03 project. 










\bibliographystyle{abbrv}
\bibliography{HoferLanger}

\end{document}